\documentclass[12pt,reqno]{amsart}

\numberwithin{equation}{section}

\usepackage{amsbsy}

\usepackage[colorlinks]{hyperref}
\hypersetup{
linkcolor=blue,          
citecolor=green,        
}

\usepackage{ifpdf}
\usepackage[nobysame,abbrev,alphabetic]{amsrefs}

\usepackage{enumerate}
\usepackage{amssymb}
\usepackage{amsmath}
\usepackage{amscd}
\usepackage{amsthm}
\usepackage{amsfonts}
\usepackage{graphicx}
\usepackage[all]{xy}
\usepackage{verbatim}
\usepackage{hyperref}
\usepackage{gensymb}
\usepackage{bm}

\newtheorem{theorem}{Theorem}[section]

\newtheorem{lemma}[theorem]{Lemma}
\newtheorem{corollary}[theorem]{Corollary}
\theoremstyle{definition}\newtheorem{definition}[theorem]{Definition}

\newtheorem{conjecture}[theorem]{Conjecture}

\theoremstyle{definition}

\theoremstyle{definition}
\theoremstyle{definition}\newtheorem{remark}[theorem]{Remark}
\theoremstyle{definition}\newtheorem*{acknowledgments}{Acknowledgments}

\newcommand{\al}{\alpha}

\newcommand{\ga}{\gamma}
\newcommand{\Ga}{\Gamma}
\newcommand{\del}{\delta}
\newcommand{\Del}{\Delta}
\newcommand{\lam}{\lambda}
\newcommand{\Lam}{\Lambda}
\newcommand{\eps}{\epsilon}

\newcommand{\sig}{\sigma}
\newcommand{\Sig}{\Sigma}

\newcommand{\Om}{\Omega}
\newcommand{\vphi}{\varphi}

\newcommand{\cF}{\mathcal{F}}

\newcommand{\cH}{\mathcal{H}}

\newcommand{\cP}{\mathcal{P}}

\newcommand{\cS}{\mathcal{S}}
\newcommand{\cT}{\mathcal{T}}

\newcommand{\bR}{\mathbb{R}}
\newcommand{\bZ}{\mathbb{Z}}
\newcommand{\bQ}{\mathbb{Q}}

\newcommand{\bN}{\mathbb{N}}

\newcommand{\bT}{\mathbb{T}}

\newcommand{\bu}{\mb{u}}
\newcommand{\bw}{\mb{w}}

\newcommand{\SL}{\operatorname{SL}}
\newcommand{\SO}{\operatorname{SO}}

\newcommand{\GL}{\operatorname{GL}}

\newcommand{\Mat}{\operatorname{Mat}}

\newcommand{\wstar}{\overset{\on{w}^*}{\to}}
\newcommand{\defi}{\overset{\on{def}}{=}}
\newcommand{\lie}{\operatorname{Lie}}
\newcommand\norm[1]{\left\|#1\right\|}
\newcommand\wt[1]{\widetilde{#1}}
\newcommand\wh[1]{\widehat{#1}}
\newcommand\set[1]{\left\{#1\right\}}
\newcommand\pa[1]{\left(#1\right)}

\newcommand\av[1]{|#1|}
\newcommand\bigav[1]{\left|#1\right|}
\newcommand\on[1]{\operatorname{#1}}
\newcommand\diag[1]{\operatorname{diag}\left(#1\right)}
\newcommand\tb[1]{\textbf{#1}}
\newcommand\mat[1]{\pa{\begin{matrix}#1\end{matrix}}}
\newcommand\br[1]{\left[#1\right]}
\newcommand\smallmat[1]{\pa{\begin{smallmatrix}#1\end{smallmatrix}}}
\newcommand\mb[1]{\mathbf{#1}}

\newcommand{\lra}{\longrightarrow}

\newcommand{\onto}{\xymatrix{\ar@{>>}[r]&}}
\newcommand{\da}[4]{\xymatrix{#1 \ar@<.5ex>[r]^{#2} \ar@<-.5ex>[r]_{#3} & #4}}

\newif\ifdraft\drafttrue

\begin{document}
\title[Equidistribution of primitive rational points]{Equidistribution of primitive rational points on expanding horospheres}
\author[M.~Einsiedler]{Manfred Einsiedler}
\author[S.~Mozes]{Shahar Mozes}
\author[N.~Shah]{Nimish Shah}
\author[U.Shapira]{Uri Shapira}
\thanks{M.~E.~acknowledges support of the SNF grant 200021-127145. S.~M.~acknowledges the support of  ISF grant 1003/11, BSF grant 2010295, and the University of Zurich. N.~S.~acknowledges the support of NSF grant 1001654. U.~S.~Chaya Fellow, acknowledges the partial support of the Advanced research Grant 228304 from the European Research Council and ISF
grant 357/13.}

\maketitle

\begin{abstract}
We confirm a conjecture of J.~Marklof regarding the limiting distribution 
of certain sparse collections of points on expanding horospheres. 
These collections are obtained by intersecting the expanded horosphere
with a certain manifold of complementary dimension and turns out to be of arithmetic nature. 
This result is then used along the lines suggested by J.~Marklof to give an analogue of a result of W.~Schmidt regarding the distribution of shapes of lattices orthogonal to integer vectors. 
\end{abstract}

\section{Main results}\label{section one} 

\subsection{The main theorem}\label{tmnt}
Given integers $m\ge n\ge 1$, let $d\defi n+m$ and consider the space $X_d\defi \SL_d(\bR)/\SL_d(\bZ)$ on which $G\defi\SL_d(\bR)$ and its subgroups act.  
Unless otherwise stated, when we write an element $g\in G$ as a matrix $g=\smallmat{A&B\\ C&D}$, we shall mean that $ A,B,C,D$ represent matrices of dimensions $m\times m, m\times n, n\times m,$ and $n\times n$ respectively. We refer to these as the \textit{block components} of $g$. We shall denote the identity matrix in dimension $k$ by $I_k$  and often write just~$I=I_k$ if the dimension is clear from the context.
Similarly, $0$ will denote the zero matrix in various dimensions.
 
Consider the following subgroups of $G$:
\begin{align*}
U&\defi\set{\smallmat{I&0\\ \bu& I}:\bu\in\Mat_{n\times m}(\bR)},\\
V&\defi\set{\smallmat{I&\mb{v}\\0&I}:\mb{v}\in\Mat_{m\times n}(\bR)},\mbox{ and}\\
H&\defi\set{\smallmat{h&\tb{v}\\0& I}: h\in \SL_m(\bR),\tb{v}\in \Mat_{m\times n}(\bR)}.
\end{align*}
Also define the diagonal matrices
\begin{equation*}
a(y)\defi\smallmat{y^{-\frac{n}{m}}I_m& 0\\ 0 & y I_n}\mbox{ for } y\in\bR_{>0}.
\end{equation*}
We denote by $x_0\in X_d$ the identity coset $\SL_d(\bZ)$ and set for $\mb{u}\in\Mat_{n\times m}(\bR)$
$$ x_\bu\defi \smallmat{I&0\\ \bu& I} x_0.$$
\begin{definition}
Given a subgroup $L<G$ and a point $x\in X_d$ we say that the orbit 
$Lx$ is  \textit{periodic} if it supports an $L$-invariant probability measure.
In this case we denote by $\mb{m}_{Lx}$ the unique $L$-invariant probability measure supported on $Lx$ and refer to it 
as the \textit{Haar} measure on the orbit.
\end{definition}
The orbits $Ux_0$, $Vx_0$ and $Hx_0$ are periodic and their respective Haar measures will play a prominent role in our discussion. In our discussion the cases 
$m>n$ and $m=n$ will exhibit different phenomena and in order to be able to have unified statements we use the 
following notation throughout:
$$\bm{\theta}\defi\Big\{
\begin{array}{ll} 
\mb{m}_{Vx_0} &\textrm{ if $m=n$}\\ 
\mb{m}_{Hx_0} &\textrm{ if $m>n$}.
\end{array}$$

\begin{definition}\label{def primitive}
We say that a matrix $\bu\in\Mat_{n\times m}(\bZ)$ is \textit{$k$-primitive}, where $k$ is a positive integer, 
if the reduction modulo $k$
of its columns span $\pa{\bZ/k\bZ}^n$. 
\end{definition}
Consider the finite\footnote{Note that for two integer matrices, if $\mb{u}_1\ = \mb{u}_2 \mod k$ then $x_{k^{-1}\mb{u}_1} =x_{k^{-1}\mb{u}_2}$.
} 
set
\begin{equation}\label{Pktilde}
\widetilde{\cP}_k\defi \set{(x_{k^{-1}\bu},x_{k^{-1}\bu}):\bu\textrm{ is $k$-primitive}}\subset Ux_0\times Ux_0,
\end{equation}
and let $\wt{\bm{\mu}}_k$ denote the normalized counting measure on $\wt{\cP}_k$;  i.e.\
\begin{equation}\label{wtmuk}
\wt{\bm{\mu}}_k\defi\frac{1}{|\wt{\cP}_k|}\sum_{(x,x)\in\wt{\cP}_k}\del_{(x,x)}.
\end{equation}
Finally, let us denote  $\wt{a}(k)=(e,a(k))\in G\times G$. With these notations
we can state our main result.

\begin{theorem}\label{JED}
As $k\to\infty$,  $\wt{a}(k)_*\wt{\bm{\mu}}_k\wstar\mb{m}_{Ux_0}\times \bm{\theta}$.
\end{theorem}

We remark that our main tool is~\cite{ShahgeneralizedRatner} which extends
the fundamental work of Ratner~\cite{R-annals}. We also note that some of our arguments are similar to~\cite{Mozes-Shah}.

\subsection{An application}\label{ssecapp}
Theorem~\ref{JED} is a generalization of a result by Marklof~\cite[Theorem 6]{Marklof-Inventiones}. 
Marklof's result has various applications, most notably to the 
distribution of Frobenius numbers, circulant graphs and the shapes of co-dimension 1 primitive subgroups of $\bZ^d$ 
(see~\cite{Marklof-Inventiones}, \cite{MarklofStrombergsson-CG}). 
Naturally, in each of these discussions an application 
of Theorem~\ref{JED} gives new results. We give one such application which is an analogue of
a certain equidistribution result of Schmidt~\cite{Schmidt98}. We follow closely the viewpoint of~\cite{Marklof-Inventiones}. 

Assume that 
$n=1, m\ge 2$ so that $d=m+1\ge 3$.
The quotient $Z_m\defi \SO_m(\bR)\backslash X_m$ will be referred  to as the \textit{space of shapes of $m$-dimensional lattices}. We equip $Z_m$ with the probability measure $\mb{m}_{Z_m}$ which is
by definition the image of $\mb{m}_{X_m}$ under the natural projection. 

For any integer vector $v\in\bZ^d$ let us denote by $\Lam_v\defi \bZ^d\cap \set{v}^\perp$; that is $\Lam_v$ is the lattice of integer points in the $m$-dimensional orthocomplement of $v$ in $\bR^d$. We may choose a matrix $k_v\in \SO_d(\bR)$ so that 
$k_v\Lam_v\subset \bR^m\times\set{0}\subset\bR^d$, 
and so that $k_v v$ lies on the positive half of the $d$-th coordinate axis. After normalizing the covolume of $k_v\Lam_v$ to be 1, we
obtain a point in $X_m$. As the choice of $k_v$ is only well defined up to the action of $\SO_m(\bR)$ we obtain a well defined point 
in $Z_m$ which we denote hereafter by $\br{\Lam_v}$.
There is a certain redundancy in considering $\Lam_v$ if $v\in\bZ^d$ is non-primitive (that is if it is an integer multiple of another integer vector). We therefore denote by $\wh{\bZ}^d$ the subset of \textit{primitive} integer vectors. Note that 
a vector $(u_1,\dots,u_m,k)\in\bZ^d$ is primitive if and only if the vector $(u_1,\dots,u_m)\in\bZ^m$ is $k$-primitive 
as in Definition~\ref{def primitive}. Finally, let $B_\infty\defi\set{x\in\bR^d:\norm{x}_\infty\le 1}$ and $\partial B_\infty$ denote its boundary.
As an application of Theorem~\ref{JED} we prove the following result in \S\ref{sec appf}.

\begin{theorem}\label{generalized Schmidt}
Let $\cF\subset\partial B_\infty$  be a measurable set of positive measure having a boundary in $\partial B_\infty$ of measure 0 with respect to the $m$-dimensional Lebesgue  measure on $\partial B_\infty$. For 
any positive integer $k$  let $$\eta_k\defi\frac{1}{\av {\wh{\bZ}^d\cap k \cF}}\sum_{v\in \wh{\bZ}^d\cap k \cF }\del_{\br{\Lam_v}}.$$ Then, $\eta_k\wstar\mb{m}_{Z_m}$ as $k\to\infty$.
\end{theorem}

Note that the choice of 
the $\ell_2$-norm yields a much more elegant statement as we have the following relation
$$\set{\Lam_v:v\in\wh{\bZ}^d, \norm{v}_2=k}=\set{\displaystyle^{\textrm{Primitive $m$-dimensional subgroups}}_{ \textrm{of $\bZ^d$ of covolume $= k$}} }$$    
We suggest the following.\footnote{Significant progress towards the suggested conjecture was obtained recently in~\cite{AES1}, \cite{AES2}.}
\begin{conjecture}\label{bold conj}
Let $d\geq 3$ and let $B_2(r)$ denote the Euclidean ball of radius $r>0$ in $\bR^d$. Then if $r_n\to\infty$ is a sequence of radii such that $\av{\partial B_2(r_n)\cap \wh{\bZ}^d}\to\infty$, then the collections
$\set{\br{\Lam_v}:v\in \partial B_2(r_n)\cap \wh{\bZ}^d}$ equidistribute in $Z_m$.
\end{conjecture}

\section{Outline and initial steps}\label{seciam}

\subsection{A motivating low dimensional example}\label{i.c}
\quad
We now prove Theorem~\ref{JED} in the case $n=m=1$ to which the techniques of the later parts of this paper do not apply.  
We learned the argument we give here from Jens Marklof and  to the best of our understanding this result was his reason to anticipate the validity of  Theorem~\ref{JED}. 

We identify both of the two orbits $Ux_0=\set{\smallmat{1&0\\s&1}x_0:s\in\bR}$ and $Vx_0=
\set{\smallmat{1&t\\0&1}x_0:t\in\bR}$ with $\bR/\bZ$ and use the parameters $s,t$ to describe them respectively.
We wish to analyze for which values of $y$ 
\begin{equation}\label{eqintne}
a(y)Ux_0\cap Vx_0\ne\varnothing
\end{equation}
and in case this happens, we wish to understand
the asymptotics of the joint distribution of the set
\begin{align}\label{set1}
\set{(s,t)\in (\bR/\bZ)^2: a(y)\smallmat{1&0\\s&1}x_0=\smallmat{1&t\\0&1}x_0}.
\end{align}
Following the definitions one sees that $(s,t)$ is in the set in~\eqref{set1} if and only if 
there exists $\smallmat{k&\ell\\m&n}\in\SL_2(\bZ)$  
which solves the equation
\begin{equation}\label{first eq intro}
\smallmat{y^{-1}& 0\\ ys&y}\smallmat{k&\ell\\m&n}=\smallmat{1&t\\0&1}.
\end{equation}
Equation~\eqref{first eq intro} has a solution if and only if 
the following three conditions on the variables $y,s,t$ are satisfied (note that we may assume that $s,t\in(0,1]$):
(i) $y=k$ is a positive integer, (ii) $t=\frac{\ell}{k}$ for some $1\le \ell\le k$ with $\gcd(\ell,k)=1$, and (iii) $s=\frac{\ell^*}{k},$ where 
$\ell^*$  is the number $1\le \ell^*\le k$ satisfying $\ell\cdot \ell^*= 1\on{mod} k$.

To summarize, if~\eqref{eqintne} holds then $y=k$ for some positive integer $k$ and the set in~\eqref{set1} equals
\begin{equation}\label{set3}
\set{(\ell/k, \ell^*/k)\in\pa{\bR/\bZ}^2: 1\le \ell\le k, \gcd(\ell,k)=1}.
\end{equation}
Taking into account both descriptions~\eqref{set1}, \eqref{set3} and our identification $(\bR/\bZ)^2\cong Ux_0\times Vx_0$ we see that 
this set is exactly $\wt{a}(k)\wt{\cP}_k$ which appears as the support of the measure in Theorem~\ref{JED}. If we denote the normalized counting 
measure on~\eqref{set3} by $\theta_k$ then 
the statement of Theorem~\ref{JED} is interpreted as saying that $\theta_k$ equidistributes in $(\bR/\bZ)^2$ as $k\to\infty$. 
This equidistribution
statement translates to estimating Kloosterman sums
 \begin{equation}\label{K sum}
K(a,b,k)=\sum_{\displaystyle^{1\le x\le k}_{\gcd(x,k)=1}} e^{2\pi i(ax+bx^*)}=\phi(k)\int_{(\bR/\bZ)^2} e^{2\pi i(at+bs)}d\theta_k,
\end{equation} 
where $\phi(k)$ is the Euler function.
The known estimates for Kloosterman sums (see for example ~\cite[p.\ 48 eq.\ (2.25)]{IwaniecSMAF}), imply that for any choice of $(a,b)\in\bZ^2$ not both 0, 
$\phi(k)^{-1}K(a,b,k)\to 0$ as $k\to\infty$ which establishes the desired equidistribution of the $\theta_k$'s. 

This establishes 
the case $m=n=1$ in Theorem~\ref{JED}. The main objective of this paper is to establish the case $m\ge 2$ using techniques 
from homogeneous dynamics.

\subsection{A basic observation and the structure of the proof}\label{m disc}
Let us define similarly to~\eqref{Pktilde}, \eqref{wtmuk},
\begin{equation}\label{Pk}
\cP_k\defi \set{x_{k^{-1}\bu}:\bu\textrm{ is $k$-primitive}}\subset Ux_0\subset X_d,
\end{equation}
\begin{equation}\label{muk}
\bm{\mu}_k\defi\frac{1}{\av{\cP_k}}\sum_{x\in\cP_k}\del_x.
\end{equation}

Using the fact that the $a(y)$-action is mixing on $X_d$ one can show that
that the pushed periodic orbit $a(y)Ux_0$ equidistributes in $X_d$ when $y\to\infty$, as $U$ is the expanding horospherical subgroup of $a(y)$ (see e.g.\ \cite{MargulisThesis}). On the other hand, the collection $\cP_k\subset Ux_0$ is composed of rational points $x$ for which the trajectories $\set{a(y)x}_{y>1}$ are divergent in $X_d$. It is therefore natural to investigate the tension between these two facts and to analyze the distribution of $a(y_k)\cP_k$
in $X_d$ for various choices of sequences $y_k\to\infty$. Setting~$y_k=k$ as in Theorem~\ref{JED} is natural because of the following lemma (which 
is in some sense the starting point of our discussion).\footnote{See Remark~\ref{rem1} for other natural choices of $y_k$'s.}
\begin{lemma}[Basic Lemma]\label{basic lemma}
For any positive integer $k$, $a(k)\cP_k\subset Hx_0$. 
\end{lemma}
We prove Lemma~\ref{basic lemma} towards the end of this section. It shows that we cannot expect the sequence $\wt{a}(k)_*\wt{\bm{\mu}}_k$ to equidistribute in $Ux_0\times X_d$ as any limit point of this sequence is clearly a measure\footnote{Apriori the limit measures is not even known to be a probability measure.} supported in $Ux_0\times Hx_0$.

The space $X_m=\SL_m(\bR)/\SL_m(\bZ)$ is naturally embedded in $Hx_0\subset X_d$, simply by identifying $\SL_m(\bR)$ as a subgroup of $H$ in
the obvious way. 
Throughout we alternate between thinking of $X_m$ as a subset of $X_d$ and as the space of $m$-dimensional unimodular lattices in $\bR^m$. 
When thinking of $X_m$ as the space of $m$-dimensional lattices, the identity coset $x_0$ corresponds to the lattice $\bZ^m$. 
There is a natural projection $\pi_3:Hx_{0}\to X_m$ 
defined as follows: $$\textrm{For $x=\smallmat{h&\mb{v}\\0&I}x_0\in Hx_{0}$, $\pi_3(x)\defi\smallmat{h&0\\0&I}\SL_d(\bZ)$ }$$
which corresponds to the $m$-dimensional unimodular lattice
$h\bZ^m$. 
Let 
\begin{equation}\label{nuk}
\bm{\nu}_k\defi (\pi_3)_*a(k)_*\bm{\mu}_k, 
\end{equation}
and consider the diagram with natural projection maps
\begin{equation}\label{diagram 1}
\xymatrix{
 &\pa{Ux_0\times Hx_{0},\wt{a}(k)_*\wt{\bm{\mu}}_k}\ar[dl]_{\pi_1}\ar[dr]^{\pi_2} & \\
 \pa{Ux_0,\bm{\mu}_k}& & \pa{Hx_{0},a(k)_*\bm{\mu}_k}\ar[d]^{\pi_3}\\
 & & \pa{X_m,\bm{\nu}_k}
}
\end{equation}
We describe the structure of the proof of Theorem~\ref{JED} using the diagram~\eqref{diagram 1}: 
\begin{enumerate}
\item[Step 1.] Establish the convergence
$\bm{\mu}_k\wstar \mb{m}_{Ux_0}$, which takes place on the left side of
diagram~\eqref{diagram 1}. 
\item[Step 2.] Show the convergence $\bm{\nu}_k\wstar (\pi_3)_*\bm{\theta}$, which equals $\mb{m}_{X_m}$ in case $m>n$, and the dirac measure
$\del_{x_0}$ in case $m=n$. This convergence takes place on the right side at the bottom of diagram~\eqref{diagram 1}.
\item[Step 3.]
Use the second step to establish the convergence $a(k)_*\bm{\mu}_k\wstar \bm{\theta}$ taking place on the right side of diagram~\eqref{diagram 1}.  
Many of the ideas appearing in the argument of Step 3 already appear in simplified versions in the proof of Step~1.
\item[Step 4.]
Combine Step 1 and Step 3 and use a disjointness argument to prove the convergence
$\wt{a}(k)_*\wt{\bm{\mu}}_k\wstar \mb{m}_{Ux_0}\times \bm{\theta}$ taking place at the top of diagram~\eqref{diagram 1}.
 \end{enumerate}
We comment here that one could prove Step 1 quite easily using Fourier transform arguments. We choose a different approach that 
is more compatible with the proofs of the other steps (see the footnote following Lemma~\ref{s.classification}).
\subsection{The method of proof}\label{sec mop}

From this point and on we will assume that $m\ge 2$. 
As described above in the course of the proof of Theorem~\ref{JED} we will frequently need to 
establish a convergence $\eta_n\wstar \eta$ of probability measures. It will turn out that all the measures involved are 
$\Lam$-invariant under natural actions of the group\footnote{
Note that the assumption that $m\ge 2$ is equivalent to $\Lam$ being non-trivial.} 
\[
\Lam\defi\set{\smallmat{\del_1&0\\0&\del_2}:\del_1\in\operatorname{SL}_m(\bZ),\del_2\in\operatorname{SL}_n(\bZ)}.
\]
The following is the strategy we will use for proving such a claim: 
One starts by \textit{classifying} the $\Lam$-ergodic measures and shows that there are only countably many such. Say 
$\set{\sig_i}_{i=0}^\infty$, with $\eta=\sig_0$. 
Then, given an accumulation point $\eta'$ of $\eta_n$, the ergodic decomposition of $\eta'$ is given by 
$\eta'=\sum_{i=0}^\infty c_i \sig_i$ with $c_i\ge0$ and $\sum c_i=1$. 
One then appeals to a \textit{non-accumulation} result to show that the only possibility of $c_i$ being positive is that 
$\on{supp}(\eta_n)\subset \on{supp}(\sig_i)$ for infinitely many $n$'s. One then verifies that the only $i$ for which
such an inclusion is possible is $i=0$.

The above strategy is built out of (i) a non-accumulation result and (ii) a measure classification result. 
In~\S\ref{n.a} we prove the non-accumulation result.
 In~\S\ref{Lambda action} we prove various measure classification results. In~\S\ref{proof of jed} we use the above strategy 
 to prove Theorem~\ref{JED} through Steps 1-- 4 which are described after~\eqref{diagram 1}.

\subsection{Elementary divisors}


In this subsection we collect some further preliminaries that will be needed throughout the paper.
We shall use the following theorem, the proof of which can be found in~\cite[Appendix Lemma A2, Theorem A1]{SDbriefguide}. 

\begin{theorem}[Elementary divisors]\label{e.d.t}
Let $\Sig\ne\set{0}$ be a subgroup of $\bZ^m$. Then, there is an integer $1\le r\le m$  and positive integers $\ell_1\dots \ell_r$ such that $\ell_i | \ell_{i+1}$ and such that one can find a basis
$v_1\dots v_m$ of $\bZ^m$ for which $\ell_1 v_1,\ldots, \ell_r v_r$ form a basis for $\Sig$. Furthermore, the numbers $r, \ell_i$ with this property are unique and are called the rank and the elementary
divisors of $\Sig$ with respect to $\bZ^m$ respectively.
\end{theorem}

The following restatement of Theorem~\ref{e.d.t} will be more convenient to us.
\begin{lemma}\label{r.f.lemma}
Consider the action of $\smallmat{\del_1&0\\0&\del_2}\in\Lam$ on $\bu\in\Mat_{n\times m}(\bZ)$
 given by $\smallmat{\del_1&0\\0&\del_2}\bu=\del_2\bu\del_1^{-1}$.  
 Each $\Lam$-orbit in $\Mat_{n\times m}(\bZ)$ contains an element of the form 
\begin{equation}\label{reduced form}
\mb{u}=\mat{
 \ell_1& &0&\dots&0 \\ 
      0  & \ddots&0&\dots&0 \\
   0& &\ell_n&\dots&0      
}
\end{equation}
such that the $\ell_i$'s are integers satisfying $\ell_1|\ell_2|\dots|\ell_n$. Moreover, the integers $\ell_i$ are unique up to sign.
\end{lemma}
\begin{proof}
The lemma follows from an application of Theorem~\ref{e.d.t} to the group $\Sig$ generated by the rows of $\bu$ in $\bZ^m$. 
Note that if $\on{rank}(\Sig)=r<n$ then the tuple $(\ell_1\dots\ell_n)$ ends with $n-r$ zeros. 
Note also that the only case where we cannot require that $\ell_i\ge 0$ 
is when $n=m$ and $\mb{u}$ is an invertible matrix with negative determinant.  
\end{proof}
Let us refer to the tuple $(\ell_1,\ldots,\ell_n)$ attached to the $\Lam$-orbit of $\bu$ as its \textit{elementary divisors tuple}.

As Definition~\ref{def primitive} of $k$-primitivity is invariant under the $\Lam$-action, it is clear that $\bu$ is $k$-primitive if and only if $\ell_n$ is coprime to $k$. For the proof of Lemma~\ref{basic lemma} we will use the following characterization of $k$-primitivity.
\begin{lemma}\label{lemprimc}
The matrix $\mb{u}\in\Mat_{n\times m}(\bZ)$ is $k$-primitive if and only if there exists a matrix $\gamma\in\SL_d(\bZ)$ whose bottom $n$ rows coincide with the $n\times d$ matrix $\mat{\mb{u} & kI_n}$. 
\end{lemma}
\begin{proof}
It is straightforward to show that the property described in the statement of the lemma is also invariant under the action of $\Lam$. It follows that it is enough to verify 
the validity of the statement for matrices $\bu$ in the form~\eqref{reduced form}. Clearly if  the elementary divisor $\ell_n$ is not coprime to $k$, i.e.\ if $\bu$ is not $k$-primitive, then the bottom row of $\mat{\bu&k I_n}$ is not a primitive vector in $\bZ^d$ and so the existence of $\gamma$ as in the statement is ruled out. If on the other hand $\bu$ is $k$-primitive then the elementary divisors $\ell_i$ are all coprime to $k$. It follows that there are integers $e_i,f_i$ so that
$\det\smallmat{e_i&f_i\\ \ell_i&k}=1.$ The determinant of the $d\times d$ matrix 
\begin{equation}\label{bigmatrix}
\gamma\defi\mat{
\diag{e_1\dots e_n} & 0& \diag{f_1\dots f_n}\\
0 &I_{n-m}&0\\
\diag{\ell_1\dots\ell_n}&0&\diag{k\dots k}
}
\end{equation}
equals $\prod_1^n\det\smallmat{e_i&f_i\\ \ell_i&k}=1$ (because after conjugating by a suitable permutation matrix, the matrix $\ga$
transforms into a block-diagonal matrix with diagonal blocks being equal to the $\smallmat{e_i&f_i\\ \ell_i&k}$ complemented 
by $m-n$ 1's). 
\end{proof}

\subsection{Proof of the Basic Lemma}


\begin{proof}[Proof of Lemma~\ref{basic lemma}]
Note that in order to show that a point $x=gx_0\in X_d$ lies in $Hx_{0}$, one needs to show that there exists $\ga\in \Ga$ such that $g\ga\in H$.
Let $\bu\in\Mat_{n\times m}(\bZ)$ be $k$-primitive and $a(k)x_{k^{-1}\textbf{u}}\in a(k)\cP_k$ be the corresponding point. 
By Lemma \ref{lemprimc} $\bu$ is 
 $k$-primitive  if and only if there exists a matrix $\ga^{-1}\in \Ga$ whose bottom $n$ rows are given by the rows of the $n\times d$ matrix 
$\pa{\textbf{u}\; kI_n}$. It follows that if we denote by $A_\ga,B_\ga,C_\ga,D_\ga$ the block components of $\ga$, then
\begin{align}\label{basic eq}
a(k)\smallmat{I&0\\k^{-1}\textbf{u}&I}\ga&= \smallmat{ k^{-\frac{n}{m}}I& 0\\ \mb{u} & k I} \smallmat{A_\ga& B_\ga\\ C_\ga& D_\ga}
=\smallmat{k^{-\frac{n}{m}}A_\ga & k^{-\frac{n}{m}} B_\ga \\ 0 & I}.
\end{align}
The above equation shows that the point $a(k)x_{k^{-1}\textbf{u}}$ belongs to $Hx_{0}$ as desired.
\end{proof}

\begin{remark}\label{rem0}
Equation~\eqref{basic eq} (which is an analogous to~\eqref{first eq intro}), is fundamental for our discussion and deserves some attention.
We note two things:
\begin{enumerate}
\item\label{r.4.4 i.2} 
Let $\bu$ be $k$-primitive and suppose $\ga$ solves~\eqref{basic eq}. By considering determinants 
we see that $A_\ga\in\Mat_m(\bZ)$ must have determinant $k^n$. This means that when considered as a lattice in $\bR^m$,
$\pi_3(a(k)x_{k^{-1}\bu})$ equals $k^{-\frac{n}{m}}A_\ga\bZ^m$, and up to homothety equals $A_\ga\bZ^m$ which is  a 
subgroup of index $k^n$ of $\bZ^m$. In Lemma~\ref{link to Hecke friends} we show that  the collection 
$\set{\pi_3(a(k)x_{k^{-1}\bu}):\bu\textrm{ is }k\textrm{-primitive}}$ consists of all such lattices of a given \textit{Hecke-type} (see 
Definition \ref{def ht} for terminology). 
\item\label{r.4.4 i.3} Assume again that $\bu$ is $k$-primitive and that $\ga$ solves~\eqref{basic eq}. The first
$m$ columns  of $\ga$ form a basis for the  discrete group of rank $m$ which we denote by $\Lam_{(\bu,k)}$, 
consisting of  integer vectors in the orthocomplement of the linear space spanned by the 
rows of the $n\times d$ matrix $\mat{\bu&kI}$. It follows that $\pi_3(a(k)x_{k^{-1}\bu})$, as a lattice in $\bR^m$,  is 
(up to homothety) the projection of $\Lam_{(\bu,k)}$ onto the copy of $\bR^m$ given by the first $m$-coordinates.
This is what furnishes the link with Schmidt's theorem and its strengthening given in Theorem~\ref{generalized Schmidt}.
\end{enumerate}
\end{remark}

\begin{remark}\label{rem1}
Let $y_k\to\infty$ be given. By Lemma~\ref{basic lemma} $a(y_k)\cP_k=a(y_k/k)a(k)\cP_k\subset a(y_k/k)Hx_0$. Thus, if we set $y_k=k^\al$ for some positive $\al$, then we obtain that $a(k^\al)\cP_k\subset a(k^{\al-1})Hx_0$. In the case $\al>1$, the collection $a(k^\al)\cP_k$ is therefore 
contained in the  uniformly divergent periodic orbit $a(k^{\al-1})Hx_0$, and in particular, the sequence of measures $a(k^\al)_*\bm{\mu}_k$ converges to the zero measure on $X_d$. In case $\al<1$, the collection $a(k^\al)\cP_k$ is therefore contained in the equidistributing periodic orbit $a(k^{\al-1})Hx_0$. It turns out that one can prove an analogue of Theorem~\ref{JED}
 in this context and show that for $0<\al<1$, $\wt{a}(k^\al)_*\wt{\bm{\mu}}_k\wstar \mb{m}_{Ux_0}\times \mb{m}_{X_d}$ (for any $m\ge n\ge 1$ including the case $n=m=1$). This analysis is non-trivial and we plan on elaborating on this in a future manuscript.

\end{remark}

\section{Non-accumulation}\label{n.a}
Our goal in this section is to prove a certain non-accumulation result -- Theorem~\ref{n.a.t} -- that will be used in various steps in the proof of Theorem~\ref{JED}. 
Although we do not aim at greatest generality, we still choose to state and prove the results in a somewhat abstract setting  (if only to isolate the necessary features that are needed for the result to hold). To this end, in this section (and in it only) we abandon the notation presented so
far and assume the following: Let $G$ be a real Lie group, let $\Ga<G$ be a lattice, and let $\Lam,L<G$ be closed 
subgroups with $\Lam$ being discrete and generated by finitely many $\on{Ad}$-unipotent elements. Assume furthermore that there is a decomposition
\begin{equation}\label{i.dec}
\lie(G)=\lie(L)\oplus W
\end{equation}
 such that $W$ is invariant under the action of $\Lam$ via the adjoint representation.
Let $X=G/\Ga$ and $z\in X$ be a point such that the orbit $Lz$ is periodic and $\Lam$-invariant. 
\begin{theorem}[Non-accumulation]\label{n.a.t}
Let $G,\Ga,\Lam,L,z,W$ and $X$ be as above and assume that the $\Lam$-representation on $W$ does not contain any fixed vectors. Let $P_k\subset X$ be a sequence of finite $\Lam$-invariant sets and $\mu_k$ the normalized counting measure on $P_k$. If $P_k\cap Lz=\varnothing$ for all $k$ then any
weak$^*$ accumulation point $\sig$ of $\set{\mu_k}_{k=1}^\infty$ satisfies $\sig(Lz)=0$.  
\end{theorem}
Below we use the absolute value symbol $\av{\cdot}$ to denote the usual absolute value of a real number as well as the Lebesgue measure of a set and the cardinality of a finite set. This should not cause any confusion.
In the proof of Theorem \ref{n.a.t} we will need to use some elementary properties of polynomials which we now recall. The following
lemma may be found at~\cite[Proposition 3.2.2]{KSS}.
\begin{lemma}\label{polylemma0}
For any degree $d$ there exist a constant $c_d>0$ such that for any polynomial $p:\bR\to\bR$ of degree bounded by $d$, for any interval $I\subset \bR$ 
we have that
if $\rho=\max\set{\av{p(x)}:x\in I}$ then for any $0<\eps\le\rho$
\begin{equation}\label{polybound0}
\frac{\av{\set{x\in I: \av{p(x)}\le\eps }}}{\av{I}}\le c_d\pa{\frac{\eps}{\rho}}^{\frac{1}{d}}.
\end{equation}
\end{lemma}
We deduce the following integer value version of this lemma.
\begin{lemma}\label{polylemma}
For any degree $d$ there exist a constant $c_d>0$ such that for any polynomial $p:\bN\to\bR$ of degree bounded by $d$, for any interval $J\subset \bN$  we have that
if $\rho=\max\set{\av{p(n)}:n\in J}$ then for any $0<\eps\le\rho$
\begin{equation}\label{polybound}
\frac{\av{\set{n\in J: \av{p(n)}\le\eps }}}{\av{J}}\le c_d\pa{\frac{\eps}{\rho}}^{\frac{1}{d}}+\frac{d}{\av{J}}.
\end{equation}
\end{lemma}
\begin{proof}
Let $I\subset\bR$ denote the real interval defined as the convex hull of $J$. For $\eps>0$ let $I_\eps\defi \set{x\in I: \av{p(x)}\le\eps }$ and $J_\eps\defi \set{n\in J: \av{p(n)}\le\eps }$.
The set $I_\eps$ is a disjoint union of finitely many closed intervals $I_{\eps,\ell}$ for $\ell=1... k$.  We have that
\begin{align*}
\av{ J_\eps }=& \av{{I_\eps\cap\bZ}}=\sum_\ell \av{{I_{\eps,\ell}\cap\bZ}}\\
=&\sum_{{\ell: \av{{I_{\eps,\ell}\cap\bZ}}=1}} 1 + \sum_{{\ell: \av{{I_{\eps,\ell}\cap\bZ}}>1}}\av{{I_{\eps,\ell}\cap\bZ}}\\
\le& d + \sum_{{\ell: \av{{I_{\eps,\ell}\cap\bZ}}>1}}2\av{I_{\eps,\ell}}\le d + 2 I_\eps ,
\end{align*}
where we used that if $I_{\eps,\ell}$ contains more than one integer
then 
\[
\av{{I_{\eps,\ell}\cap\bZ}}\le 2\av{I_{\eps,\ell}},
\]
and also, the number of $\ell$'s for which $I_{\eps,\ell}$ contains a single integer is bounded by the 
degree $d$ of the polynomial because between each such two intervals there must be a zero of the derivative.

The inequality~\eqref{polybound} now follows from~\eqref{polybound0} (with a slightly bigger constant $c_d$).
\end{proof}


\begin{proof}[Proof of Theorem~\ref{n.a.t}]
Let $\sig$ be a weak$^*$ accumulation point as in the statement. 
Let $\Om_1\subset X$ be a compact set. 
We will show that for $K_1\defi \Om_1\cap Lz$ one has $\sig(K_1)=0$. This 
is enough as $Lz$ is a countable union of such sets. 
Choose some norm on $\lie(G)$ and denote by $B_\eps^W$ the ball of radius $\eps$ around $0$ in $W$.
Choose an open set $\wt{\Om}_1$ and a compact set $\Om_2$ such that (i) $\Om_1\subset\wt{\Om}_1\subset \Om_2$, (ii) 
$\wt{\Om}_1\cap\pa{\exp B_1^W\cdot\pa{X\smallsetminus\Om_2}}=\varnothing$; in other words, $\Om_2$ is big enough so that one cannot reach $\wt{\Om}_1$ by acting on points outside of 
$\Om_2$ by elements of the form
$\exp w$, where $w\in W$ is of norm $\le 1$. Let $\wt{K}_1\defi \wt{\Om}_1\cap Lz$ and $K_2\defi \Om_2\cap Lz$ and for any $\eps>0$ any subset $F\subset  Lz$ denote $\cT(\eps,F)\defi\set{\exp (w) x:x\in F, w\in W\mbox{, and }\norm{w}\le\eps}$ the $\eps$-\textit{tube around} $F$. 

There exist $0<\eps_0<1/2$ small enough so that the map 
$(w,x)\mapsto \exp(w) x$ from $B_{\eps_0}^W\times K_2\to \cT(\eps_0,K_2)$ is a homeomorphism onto
its image. This gives a natural coordinate system on the $\eps_0$-tube around $K_2$; we denote for $y\in \cT(\eps_0, K_2)$ by $w_y\in B_{\eps_0}^W$ the 
$W$-\textit{coordinate} and by $x_y\in Lz$ the \textit{orbit-coordinate} so that the identity $y=\exp(w_y) x_y$ holds for $y\in \cT(\eps_0, K_2)$.

Let $u_1,\dots,u_r\in \Lam$ be $\on{Ad}$-unipotent elements that generate $\Lam$. 
Let $0\le \eps< \eps_0$ and set
$$\cS_j(\eps, \wt{K}_1)\defi\set{y\in\cT(\eps, \wt{K}_1): \on{Ad}_{u_j}(w_y)\ne w_y},$$ $j=1\dots r$, so that 
$\cT(\eps,\wt{K}_1)\smallsetminus \wt{K}_1= \cup_1^r\cS_j(\eps,\wt{K}_1)$. The inclusion $\supset$ is clear. The other 
inclusion holds because 
$\Lam$ is generated by the $u_j$'s and $W$ contains no $\Lam$-fixed non-zero vectors. We will find
a function $\psi(\eps)\to_{\eps\to0} 0$ such that for any $j$
\begin{equation}\label{r.est}
\frac{\av{P_k\cap \cS_j(\eps,\wt{K}_1)}}{\av{P_k}}\le \psi(\eps)\quad \textrm{for any $k$ }.
\end{equation}
Since $P_k\cap Lz = \varnothing$ this implies that 
$$\frac{\av{P_k\cap \cT(\eps, \wt{K}_1)}}{\av{P_k}}\le\sum_1^r\frac{\av{P_k\cap \cS_j(\eps,\wt{K}_1)}}{\av{P_k}}\le r\psi(\eps)\to_{\eps\to 0}0.$$ In turn, this implies that $\sig(K_1)=0$ as desired because $\cT(\eps, \wt{K}_1)$ is an open set containing the compact set $K_1$.  

To this end, fix $1\le j\le r$  and denote for each $y\in \cS_j(\eps,\wt{K}_1)$, 
$$p_y(n)\defi \norm{\on{Ad}_{u_j}^n(w_y)}^2.$$ 
Note that
as $u_j$ is $\on{Ad}$-unipotent it follows that (for an appropriate choice of a norm $\norm{\cdot}$), $p_y(n)$  is a non-constant polynomial in $n$ of degree
$\le d$ for some integer $d$ depending on $\dim G$ only.
Let us use the following notation
\begin{itemize}
\item $n_y\defi \max\set{n\ge 0: p_y(k)\le \eps_0^2\textrm{ for } 0\le k\le n}$.
\item $J_y\defi [0,n_y]\cap \bZ$.
\item $V_y\defi\set{u_j^ny: n\in J_y}$.
\item $M\defi \max_{j=1\dots r} \norm{\on{Ad}_{u_j}^{\pm 1}}$.
\end{itemize} 
Observe that $n_y$ is finite as $p_y(n)$ is non-constant. 
We will shortly show that the following properties hold for $y,y_1,y_2\in \cS_j(\eps, \wt{K}_1)$.
\begin{enumerate}
\item\label{n.a.p.1}$V_{y_1}\cap V_{y_2}\ne \varnothing\Longrightarrow V_{y_1}\subset V_{y_2}$ or $V_{y_2}\subset V_{y_1}$.
\item\label{n.a.p.2}  $V_y\cap \cS_j(\eps, \wt{K}_1)\subset\set{u_j^n y: p_y(n)\le\eps^2}$.
\item\label{n.a.p.2.4}  $p_y(n_y)\ge (\frac{\eps_0}{M})^2.$
\item\label{n.a.p.2.5}  $\av{J_y}\ge \log(\frac{\eps_0}{\eps})/\log M.$
\end{enumerate}
We now conclude the proof using properties~\eqref{n.a.p.1}-\eqref{n.a.p.2.5}.
Given any $k$, we choose a finite collection  $\set{y_i}\subset P_k\cap \cS_j(\eps, \wt{K}_1)$  so that the $V_{y_i}$'s are 
maximal with respect to inclusion among $\set{V_y:y\in P_k\cap \cS_j(\eps, \wt{K}_1)}$, and such that $P_k\cap \cS_j(\eps,\wt{K}_1)\subset \bigcup_i V_{y_i}$.
By property~\eqref{n.a.p.1} we deduce that the $V_{y_i}$'s are disjoint. Since $V_{y_i}\subset P_k$, we deduce that $\sum_i\frac{\av{J_{y_i}}}{\av{P_k}}\le 1$. 
It follows that
\begin{align*}
\frac{\av{P_k\cap \cS_j(\eps,\wt{K}_1)}}{\av{P_k}}&= \frac{\av{\cup_i V_{y_i}\cap \cS_j(\eps,\wt{K}_1)}}{\av{P_k}}
\le \sum_i \frac{\av{J_{y_i}}}{\av{P_k}}\cdot \frac{\av{V_{y_i}\cap \cS_j(\eps, \wt{K}_1)}}{\av{J_{y_i}}}
\\
&\overset{\textrm{{\tiny by \eqref{n.a.p.2}}}}
{\le} \sum_i \frac{\av{J_{y_i}}}{\av{P_k}}\cdot \frac{\av{\set{n\in J_{y_i}: p_{y_i}(n)
\le \eps^2}}}{\av{J_{y_i}}}
\\
&\overset{\textrm{{\tiny by \eqref{n.a.p.2.4},\eqref{n.a.p.2.5} and \eqref{polybound}}}}
{\le} c_d \pa{\frac{\eps}{\eps_0/M}}^{2/d} + \frac{d}{\log(\frac{\eps_0}{\eps})/\log M}.
\end{align*}
As this last expression goes to $0$ as $\eps\to 0$ we conclude that~\eqref{r.est} holds
with this last expression as $\psi(\eps)$.
It remains to verify the validity of~\eqref{n.a.p.1}-\eqref{n.a.p.2.5}.

\eqref{n.a.p.1}. Assume $y_1,y_2\in \cS_j(\eps,K_1)$ are such that $V_{y_1}\cap V_{y_2}\ne\varnothing$; 
that is, there exists $n_i\in J_{y_i}$ such that $u_j^{n_1} y_1=u_j^{n_2}y_2$. Assume without loss of generality that $n_2\le n_1$ and so $u_j^{n_1-n_2}y_1=y_2$
and $n_1-n_2\in J_{y_1}$. Following the definitions we  see that $V_{y_2}\subset V_{y_1}$ as desired. 

\eqref{n.a.p.2}. Let $y\in \cS_j(\eps, \wt{K}_1)$ and 
assume $n\in J_y$ is such that $p_y(n)>\eps^2$ so that we know that
$\eps < \norm{\on{Ad}_{u_j}^n(w_y)}\le \eps_0$.
We need to show that $u_j^ny\notin  \cS_j(\eps, \wt{K}_1)$.
We have that 
\begin{equation}\label{n.a.identity}
u_j^n y=u_j^n\exp(w_y) u_j^{-n} u_j^nx_y =  \exp\pa{\on{Ad}_{u_j}^n(w_y)} u_j^n x_y.
\end{equation}  
If $u_j^n x_y\notin \Om_2$ then~\eqref{n.a.identity} implies that $u_j^n y\in \exp B_{\eps_0}^W(X\smallsetminus\Om_2)$ which is
disjoint from $\cT(\eps_0, \wt{K}_1)$ by choice of $\Om_2$. So in particular, $u_j^ny\notin  \cS_j(\eps, \wt{K}_1)$. If on the other hand $u_j^n x_y\in \Om_2$, then as $u_j^nx_y\in Lz$ (because $Lz$ is $\Lam$-invariant), we deduce that $u_j^n x_y\in K_2$ which in turn implies by~\eqref{n.a.identity} that $u_j^ny\in\cT(\eps_0, K_2)$ and the orbit and $W$ coordinates of $u_j^ny$ are given by $ \on{Ad}_{u_j}^n(w_y)$ and 
$u_j^n x_y$ respectively. By the lower bound on the $W$-coordinate we deduce that $u_j^n y\notin \cT(\eps, K_2)$
and in particular, $u_j^n y\notin \cS_j(\eps, \wt{K}_1)$.  

\eqref{n.a.p.2.4}. We have that 
$\eps_0\le  \norm{\on{Ad}_{u_j}^{n_y+1}(w_y)}\le M\sqrt{p_y(n_y)}.$

\eqref{n.a.p.2.5}. Similarly $\eps_0\le  \norm{\on{Ad}_{u_j}^{n_y+1}(w_y)}\le M^{n_y+1}\eps.$

\end{proof}
\begin{corollary}\label{r.negligible}
The conclusion of Theorem~\ref{n.a.t} remains valid if the assumption $P_k\cap Lz=\varnothing$ is relaxed to $\frac{\av{P_k\cap Lz}}{\av{P_k}}\to0$ as $k\to\infty$.
\end{corollary}
\begin{proof}
We split $\mu_k=(1-\al_k) \mu_k^1+\al_k\mu_k^2$ with $\mu_k^i$, $i=1,2$ being the normalized counting measures on 
$P_k\smallsetminus Lz$ and $P_k\cap Lz$ respectively. In this case $\al_k=\frac{\av{P_k\cap Lz}}{\av{P_k}}\to0$ as $k\to\infty$ by
assumption and so the accumulation points of $\mu_k$ are the same as those of $\mu_k^1$ for which Theorem~\ref{n.a.t} applies.
\end{proof}

\section{$\Lam$-invariance and measure classifications}\label{Lambda action}

In this section we will show that all the measures appearing in our discussions  are invariant under a certain group $\Lam$. We will
then classify all the $\Lam$-invariant and ergodic probability measures in certain situations. This will serve us in the proof of Theorem~\ref{JED} along the lines described in~\S\ref{sec mop}.

\subsection{Invariance}
We return to use the notation introduced in \S\ref{section one}, \S\ref{seciam}.
In particular, recall that the subgroup $\Lam\simeq\SL_m(\bZ)\times\SL_n(\bZ)<\SL_d(\bZ)$ 
is defined by
\begin{equation}\label{Lambda}
\Lam\defi\set{\smallmat{\del_1&0\\0&\del_2}: \del_1\in\SL_m(\bZ),\del_2\in\SL_n(\bZ)}.
\end{equation}
Further, let $\Lam_\Del<G\times G$ denote the diagonal embedding of $\Lam$ in $G\times G$. 
\begin{lemma}\label{invariance under Lambda}
The following periodic orbits in either $X_d$ or $X_d\times X_d$ 
and probability measures supported on them are $\Lam$-invariant or $\Lam_\Del$-invariant respectively.
\begin{enumerate} 
\item\label{invlem1} The periodic orbit $Ux_0$ and the measures $\mb{m}_{Ux_0},\set{\bm{\mu}_k:k\in\bZ_{>0}}$.
\item\label{invlem2} The periodic orbit $Hx_0$ and the measures $\mb{m}_{Hx_0}$, $\set{a(k)_*\bm{\mu}_k:k\in\bZ_{>0}}$.
\item\label{invlem3} The periodic orbit $Vx_0$ and the measure $\mb{m}_{Vx_0}$. 
\item\label{invlem4} The periodic orbit $X_m$ and the measures $\mb{m}_{X_m}$,$\set{\bm{\nu}_k:k\in\bZ_{>0}}$ .
\item\label{invlem5} The periodic orbit $Ux_0\times Hx_0$ and the measures $\mb{m}_{Ux_0}\times \mb{m}_{Hx_0}$, 
$\set{\wt{a}(k)_*\wt{\bm{\mu}}_k:k\in\bZ_{>0}}$.
\item\label{invlem6} The periodic orbit $Ux_0\times Vx_0$ and the measure $\mb{m}_{Ux_0}\times \mb{m}_{Vx_0}$.
\end{enumerate}
\end{lemma} 
\begin{proof}

The $\Lam$-action on $Ux_0$, $Hx_0$ and $X_m$ are given by  
\begin{equation}\label{action1}
\smallmat{\del_1&0\\ 0&\del_2}\smallmat{I&0\\ \bu&I}x_0=\smallmat{I&0\\ \del_2\bu\del_1^{-1} &I}x_0.
\end{equation}
\begin{equation}\label{action2}
\smallmat{\del_1&0\\ 0&\del_2}\smallmat{h&\mb{v}\\ 0&I}x_0 = \smallmat{\del_1h &\del_1 \mb{v}\del_2^{-1}\\ 0 & I}x_0.
\end{equation}
\begin{equation*}
\smallmat{\del_1&0\\0&\del_2}\smallmat{h&0\\0&I}x_0=\smallmat{\del_1h&0\\0&I}x_0. 
\end{equation*}
The $\Lam$-action on $Vx_0$ is given by a similar formula. This shows that the corresponding periodic orbits are 
indeed $\Lam$-invariant. Also, as conjugation by $\Lam$ fixes the volume form on the groups giving rise to these
periodic orbits, the Haar measures on these periodic orbits are preserved. 

The measures $\wt{\bm{\mu}}_k$ are $\Lam_\Del$-invariant 
because it follows from~\eqref{action1} and Definition~\ref{def primitive} that $\wt{\cP}_k$ is $\Lam_\Del$-invariant. In turn, because the 
$\Lam_\Del$-action on $X_d\times X_d$ commutes with that of $\wt{a}(k)$, we conclude that 
$\wt{a}(k)_*\wt{\bm{\mu}}_k$ is $\Lam_\Del$-invariant. Similarly $a(k)_*\bm{\mu}_k$ is $\Lam$-invariant. 
The invariance of the measure $\bm{\nu}_k$ now follows from that of $\bm{\mu}_k$ as the projection 
$\pi_3$ in~\eqref{diagram 1} intertwines the $\Lam$-actions on $Hx_0$ and
$X_m$. 
\end{proof}
\subsection{Rationality issues}\label{ssecri}
We will need the following lemmas in order to establish various rationality statements when classifying measures. These 
rationality statements are important to us because they imply the countability of the measures we classify in each discussion. This countability is used later along the lines described in~\S\ref{sec mop}.
\begin{lemma}\label{rationalitylemma0}
Let $N$ be an integer and let $\lam\in\SL_N(\bZ)$ be matrix acting naturally on the torus $\bT^N=\bR^N/\bZ^N$. Assume
that all the eigenvalues of $\lam$ are not roots of unity. Then, if $\mb{w}\in \bT^N$ has a finite $\lam$-orbit then $\mb{w}$ is
a rational point (that is, any vector representing it is rational).
\end{lemma}
\begin{proof}
Assume that $\lam^j\bw=\bw$. If $w\in\bR^N$ projects to $\mb{w}$ then this means that there is an integer vector $e$ such that
$\lam^j w= w+e$ or said otherwise $(\lam^j-I)w=e$. By assumption $\lam^j-I$ is invertible and its inverse is a rational matrix so
$w=(\lam^j-I)^{-1}e$ is rational as well.
\end{proof}
We recall the definition of the commensurator group: Let $G'$ be a topological group and $\Ga'<G'$ a closed subgroup. 
Let  $\on{comm}_{G'}(\Ga')\defi \set{g\in G':g\Ga' g^{-1}, \Ga'\textrm{ are commensurable}}.$
\begin{lemma}\label{comlem}
Let $G'$ be a topological group and $\Ga',\Lam'$ closed subgroups. 
\begin{enumerate}
\item\label{comlem1} Let $g\in G'$ and $q\in\on{comm}_{G'}(\Ga')$. Then, if the orbit $\Lam g\Ga'\subset G'/\Ga'$ is finite, then so is $\Lam gq\Ga'$.
\item\label{comlem2} If $\Ga'<G'$ is a lattice then $q\in\on{comm}_{G'}(\Ga')$ if and only the orbit of $\Ga'g\Ga'\subset G'/\Ga'$ is finite.
\end{enumerate}
\end{lemma}
\begin{proof} 
\eqref{comlem1}. Note that $\Lam' g\Ga'\subset G'/\Ga'$ is finite if and only if $\on{Stab}_{\Lam'}g\Ga'=\Lam'\cap g\Ga' g^{-1}<\Lam'$ is of finite index. As
$q\in\on{comm}_{G'}(\Ga')$ we deduce that the intersection $gq\Ga'q^{-1}g^{-1}\cap g\Ga' g^{-1}$ is of finite index in both groups. In particular, 
$\Lam'\cap gq\Ga'q^{-1}g^{-1}<\Lam'$ is of finite index. 
Arguing in reverse this implies  now that $\Lam'gq\Ga'\subset G'/\Ga'$ is finite.

\eqref{comlem2}. One direction of implication follows by applying part~\eqref{comlem1} to $\Lam'=\Ga'$. In the other direction, if the orbit $\Ga'q\Ga'$ is finite
then as before, this implies that the group $\on{Stab}_{\Ga'}q\Ga'=\Ga'\cap q\Ga' q^{-1}$ is of finite index in $\Ga'$. In particular, it is a lattice in $G'$, which
forces its index in $q\Ga' q^{-1}$ to be finite as well; i.e.\ $q\in\on{comm}_{G'}(\Ga')$.
\end{proof}
\begin{lemma}\label{rationalitylemma}
If $\smallmat{g&\mb{v}\\0&I}x_0$ has a finite $\Lam$-orbit then
\[
 g\in\on{comm}_{\SL_m(\bR)}(\SL_m(\bZ))\mbox{ and }\mb{v}\mbox{ is rational}.
\]
A similar statement holds for 
$\smallmat{g&0\\\mb{u}&I}x_0$.
\end{lemma}
\begin{proof}
We prove the first statement. The second statement follows by applying the involution on 
$X_d$ induced by the transpose inverse operation.  
Assume that $\Lam \smallmat{g&\mb{v}\\0&I}x_0$ is finite. Projecting to $X_m\cong \SL_m(\bR)/\SL_m(\bZ)$ we deduce by Lemma~\ref{comlem}\eqref{comlem2}
that 
\[
  g\in\on{Comm}_{\SL_m(\bR)}(\SL_m(\bZ)).
\]
 In turn, this implies that $q=\smallmat{g&0\\ 0&I}\in\on{comm}_{G}(\SL_d(\bZ))$. It now 
follows from Lemma~\ref{comlem}\eqref{comlem1} that $\Lam\smallmat{g&\mb{v}\\0&I}q^{-1}x_0=\Lam\smallmat{I&\mb{v}\\0&I}x_0$ is finite as well. Thus we have reduced the 
discussion to the situation where $g=I$. 

We now identify $Vx_0$ with the torus $\bR^N/\bZ^N$ with $N=m\cdot n$ and apply Lemma~\ref{rationalitylemma0} to conclude that $\mb{v}$ is rational. Indeed it is straightforward to verify the existence of $\lam\in \Lam$ that acts on $\bR^N$ without roots of unity as eigenvalues.
\end{proof}

\subsection{Measure classifications in general}
\begin{theorem}\label{NimishRatner1}
Let $\mu$ be a $\Lam$-invariant and ergodic probability measure on $X_d$. Then there exists an intermediate subgroup $\Lam<L<G$ and a periodic $L$-orbit $Lx\subset X_d$, such that $\mu$ is the $L$-invariant probability measure $\mb{m}_{Lx}$.
\end{theorem}
This theorem is a particular case of a more general measure classification by Shah~\cite{ShahgeneralizedRatner} 
which  uses and generalizes Ratner's measure classification 
theorem \cite{R-annals}, \cite{R-Duke}.
It is applicable since 
 $\Lam=\SL_m(\bZ)\times \SL_n(\bZ)$ is generated by unipotent elements.
 In fact, because $\Lam$ is a lattice in  a semisimple Lie subgroup of $G$ 
with no compact factors, using the suspension technique~\cite[Corollary 5.8]{WitteMQ} it is straightforward to deduce 
Theorem~\ref{NimishRatner1} directly from Ratner's measure classification theorem for the actions of semisimple groups without compact factors (for a simplified proof for this case see~\cite{ManfredSemisimpleRatner}).

For a closed subgroup $L<G$ we denote by $L^\circ$ the 
connected component of the identity of $L$.
We have the following corollary which will be more convenient for us.
\begin{corollary}\label{minor difference}
Let $\mu$ be a $\Lam$-invariant and ergodic probability measure on $X_d$ and let $L, x$ be the group and the point that arise by applying Theorem~\ref{NimishRatner1} so that $\mu = \mb{m}_{Lx}$. 
Then there exist $x_1\dots x_N\in Lx$ such that $Lx=\bigsqcup_{i=1}^N L^\circ x_i$, each orbit $L^\circ x_i$ is periodic, $\Lam$ acts transitively by permuting the collection of orbits $\set{L^\circ x_i}$, and $\mu=\frac{1}{N}\sum_1^N \mb{m}_{L^\circ x_i}.$
\end{corollary}
\begin{proof}
As $L^\circ\unlhd L$ is open and closed, the orbit $Lx$ decomposes
into a union of (relatively) open and closed $L^\circ$-orbits. As $\mu$ is finite, this decomposition is finite $Lx=\bigsqcup_{i=1}^N L^\circ x_i$. As $L^\circ$ is a normal subgroup of $L$, $\Lam$ acts on the $L^\circ$-orbits by permuting them. Moreover, the
ergodicity assumption implies that this action is transitive. It follows that $\mu(L^\circ x_i)=N^{-1}$ and the formula 
$\mu=\frac{1}{N}\sum_1^N \mb{m}_{L^\circ x_i}$ follows.
\end{proof}

For convenience of reference we also state the following elementary lemma whose proof we omit.  
\begin{lemma}\label{groups and orbits1}
Let $F_1,F_2$ be closed subgroups of $G$ and $x\in X_d$. If $F_1x\subset F_2x$ then $F^\circ_1\subset F_2$.
\end{lemma}
The rest of this section is devoted to classifying $\Lam$-invariant and ergodic measures in various situations that will be encountered 
in the course of the proof of Theorem~\ref{JED}.
\subsection{Measures supported in $Ux_0$}

Let us define
\begin{equation}\label{Tor}
\on{Tor}_k(Ux_0)\defi\set{x_{k^{-1}\bu}\in Ux_0: \bu\in\Mat_{n\times m}(\bZ)}.
\end{equation}
\begin{lemma}\label{s.classification}
The ergodic $\Lam$-invariant probability measures on $Ux_0$ are exactly the normalized counting measures on finite $\Lam$-orbits and $\mb{m}_{Ux_0}.$ Moreover, any finite $\Lam$-orbit is contained in $\on{Tor}_k(Ux_0)$ for some positive integer $k$.
\end{lemma}

\begin{proof}
Let\footnote{Lemma~\ref{s.classification} could be proved using Fourier arguments but we choose to appeal to  
Theorem~\ref{NimishRatner1} as this is more compatible with the later arguments. Furthermore, this is by no means a new result and the 
proof is included for expository reasons. For a similar argument see for example~\cite[Example 5.9]{WitteMQ}.} $\mu$ be an ergodic $\Lam$-invariant measure supported in the orbit $Ux_0$. Applying Theorem~\ref{NimishRatner1} we
conclude the existence of a closed subgroup $\Lam<L<G$ such that $\mu$ is the $L$-invariant probability measure supported on a periodic $L$-orbit. 
Applying  Lemma~\ref{groups and orbits1} we conclude that $L^\circ< U$. Viewing $L^\circ$ as a subspace of
$\on{Mat}_{n\times m}(\bR)\cong U$, the fact that $\Lam$ normalizes $L^\circ$ translates into this subspace being invariant
under the linear representation of $\Lam$ on $\on{Mat}_{n\times m}(\bR)$ 
(here $\smallmat{A&0\\0&D}\in\SL_m(\bR)\times\SL_n(\bR)$
acts on $\mb{u}\in \on{Mat}_{n\times m}(\bR)$ by $D\mb{u}A^{-1}$). As this representation of $\SL_m(\bR)\times\SL_n(\bR)$ is
irreducible and $\Lam$ is Zariski dense in the former group, we deduce that $L^\circ$ is either the trivial group or $U$. 

If $L^\circ= U$ then clearly $\mu=\mb{m}_{Ux_0}$. If $L^\circ$ is trivial, an application
of Corollary~\ref{minor difference} gives that 
$\mu$ is the normalized counting measure on a finite $\Lam$-orbit $\Lam x_\mb{u}\subset Ux_0$. Lemma~\ref{rationalitylemma} implies now that
 $\mb{u}$ is in fact a rational matrix. If $k$ is a common denominator for its entries then $x_{\mb{u}}\in\on{Tor}_k(Ux_0)$ 
as desired.
\end{proof}

\subsection{Measures supported in $Hx_0$}\label{msiHx}

 Our next objective is to classify the $\Lam$-invariant ergodic probability measures
supported in $Hx_{0}$.  By Theorem~\ref{NimishRatner1} and 
Corollary~\ref{minor difference} such a measure is always of the form 
$\frac{1}{N}\sum \mb{m}_{L x_i}$, where  $L<G$ is a closed connected subgroup normalized by $\Lam$.
As $Lx_i\subset Hx_0$ and $L$ is connected we conclude from Lemma~\ref{groups and orbits1} that $L<H$.
We have the following classification of groups. 
\begin{lemma}\label{classifying groups}
Let $L<H$ be a closed connected subgroup normalized by $\Lam$. Then there are four possibilities
\begin{enumerate}
\item\label{c.g.1} $L=\set{e}$,
\item\label{c.g.2} $L=V=\set{\smallmat{I&\mb{v}\\0&I}:\mb{v}\in\Mat_{m\times n}(\bR)},$
\item\label{c.g.3} $L=\SL_m(\bR)=\set{\smallmat{h&0\\0&I}:h\in\SL_m(\bR)},$
\item\label{c.g.4} $L=H$.
\end{enumerate}
\end{lemma}
\begin{proof}
Consider the projection of $L$ in the simple group $\SL_m(\bR)=H/V$. As $L$ is normalized by $\Lam$, this projection is a connected subgroup normalized by $\SL_m(\bZ)$ and by Zariski density we deduce that it is a connected normal subgroup. It therefore follows that this projection is
either trivial or $\SL_m(\bR)$. In the first case, $L<V$. As the adjoint action of $\Lam$ on the Lie algebra of $V$ (which is isomorphic to $V$ itself) is irreducible, it follows that either $L=V$ or $L$ is the trivial group. 

Assume then that the projection of $L$ is onto and consider the subgroup $V'=L\cap V$
which is normalized by $\Lam$. The same irreducibility argument as before implies that either $V'=V$ or $V'$ is trivial.
In the first case $L=H$. In the second case 
the projection of $L$ onto $H/V$ must also be injective so 
there exists a map $\psi:\SL_m(\bR)\to\Mat_{m\times n}(\bR)$ such that $L=\set{\smallmat{g&\psi(g)\\ 0&I}:g\in\SL_m(\bR)}$. 
If $\lam=\smallmat{\del&0\\0&I}\in \Lam$ then $\lam\smallmat{\del&\psi(\del)\\ 0&I}\lam^{-1}=\smallmat{\del&\del\psi(\del)\\ 0&I}\in L$ so that $\del\psi(\del)=\psi(\del)$. By Zariski density we deduce that this formula holds for any $\del\in \SL_m(\bR)$. In turn, if $1$
is not an eigenvalue of $\del$, this implies that the columns of $\psi(\del)$ must be zero. Continuity now gives that $\psi$ vanishes and so $L$ is the standard copy 
of $\SL_m(\bR)$ in $H$.
\end{proof}
In light of  Lemma~\ref{classifying groups} it makes sense to make (for the sake of the current discussion), the following
\begin{definition}\label{t.o.m}
We say that a $\Lam$-invariant  and ergodic measure on $Hx_{0}$ is of \textit{type} (1)--(4) according to which one of the four groups that appear in Lemma~\ref{classifying groups}
is attached to it.
\end{definition}
\begin{corollary}\label{summarizing}
The following is a classification of the $\Lam$-invariant and ergodic probability measures on $Hx_{0}$.
\begin{enumerate}
\item Measures of type~\eqref{c.g.1} are simply normalized counting measures on a (necessarily finite) $\Lam$-orbit of a point of the form $\smallmat{h&\mb{v}\\ 0&I}x_0$, where 
$h\in\on{comm}_{\SL_m(\bR)}(\SL_m(\bZ))$ and $\mb{v}$ is rational. If such a measure projects to $\del_{x_0}$ under $(\pi_3)_*$ then we can choose $h=I$.
\item\label{sum2} Measures of type~\eqref{c.g.2} are of the form $\frac{1}{N}\sum_{i=1}^N \mb{m}_{Vx_i}$, where $\set{x_i}$ is the $\Lam$-orbit of $x_1 =\smallmat{h&0\\0&I}x_0$, where $h\in\on{comm}_{\SL_m(\bR)}(\SL_m(\bZ))$. There is only one such measure that projects to $\del_{x_0}$ and that is $\mb{m}_{Vx_0}$.
\item Measures of type~\eqref{c.g.3} are of the form $\frac{1}{N}\sum_1^N\mb{m}_ {\SL_m(\bR)x_i}$, where $x_i=\smallmat{I&\mb{v}_i\\ 0& I}x_0$ 
and the $\mb{v}_i$'s are
rational. All such measures project under $(\pi_3)_*$ to $\mb{m}_{X_m}$.
\item There is only one measure of type~\eqref{c.g.4}, namely $\mb{m}_{Hx_{0}}$ and it projects under $(\pi_3)_*$ to $\mb{m}_{X_m}$.
\end{enumerate}
\end{corollary}
\begin{proof}
Let $\mu$ denote a $\Lam$-invariant and ergodic probability measure on $Hx_0$ and present it 
as $\mu=\frac{1}{N}\sum_1^N\mb{m}_{Lx_i}$ with $L$ connected and determining the type as explained above. We prove statements (1)--(3) as (4) is clear.

\noindent (1). If $\mu$ is of type~\eqref{c.g.1} we deduce from the ergodicity that $\set{x_i}_1^N$ forms a finite $\Lam$-orbit 
and Lemma~\ref{rationalitylemma} gives the commensurability and rationality statements. The statement regarding the projection is clear.

\noindent (2). Suppose $\mu$ is of type~\eqref{c.g.2} and write $x_i=\smallmat{h_i&\mb{v}_i\\0&I}x_0$. Using the $V$-invariance we 
may assume without loss of generality that $\mb{v}_i=0$ and so $\pi_3(x_i)=x_i$. Then 
$(\pi_3)_*\mu=\frac{1}{N}\sum_1^N\del_{x_i}$ is a finitely
supported $\Lam$-invariant and ergodic measure and so the $x_i$'s form an orbit. Lemma~\ref{rationalitylemma} gives now that $h_1\in\on{comm}_{\SL_m(\bR)}(\SL_m(\bZ))$ as required.

\noindent (3).  Assume $\mu$ is of type~\eqref{c.g.3}. As the periodic orbits $\SL_m(\bR)x_i$ are closed and the $\SL_m(\bR)$ orbits 
are transverse to the $V$-orbits (all of which are periodic in $Hx_0$), we deduce that  the intersection 
$\on{supp}(\mu)\cap Vx_0$, which is $\Lam$-invariant and closed, is finite. Lemma~\ref{rationalitylemma} now implies that 
the points $x_i$ which without loss of generality might be assumed to belong to $\on{supp}(\mu)\cap Vx_0$, have rational representatives.
\end{proof}
\subsection{Measures supported in $Ux_0\times Hx_0$}
\begin{theorem}\label{mcn<m}
The only $\Lam_\Del$-invariant probability measure on $Ux_0\times Hx_0$ which projects under $\pi_1$ and $\pi_2$ to $\mb{m}_{Ux_0}$, $\mb{m}_{Hx_0}$ respectively, is the product measure  $\mb{m}_{Ux_0}\times \mb{m}_{Hx_0}$.
\end{theorem}
\begin{proof}
Let $u\in \Lam$ be a unipotent element that belongs to the subgroup $\SL_m(\bZ)<\Lam$ (such an element always exists as $m\ge 2$). 
Note the following two facts
\begin{enumerate}
\item  The action of $u$ on $ (Hx_{0},\mb{m}_{Hx_{0}})$ is mixing by the Howe-Moore theorem applied to the action of
$\SL_m(\bR)$.
\item The action of $u$ on the torus $\bT^{n\times m}\cong Ux_0$ is by a unipotent automorphism and so the ergodic components of $\mb{m}_{Ux_0}$ with respect to the action
of $u$ are minimal rotation on compact abelian groups.
\end{enumerate}
The theorem now follows from the following two lemmas. 
\end{proof}
\begin{lemma}[cf.\ {\cite[Theorem 6.27]{GlasnerBook}}]\label{disjoint}
Let $(X,\mu,u),(Y,\nu,u)$ be two dynamical systems. If $(X,\mu,u)$ is 
a minimal rotation on a compact abelian group and $(Y,\nu,u)$ is mixing, then the two systems are disjoint.
\end{lemma}
\begin{lemma}\label{e.c.r}
Let $(X,\mu,u),(Y,{\nu},u)$ be two dynamical systems such that ${\nu}$ is ergodic and let 
$\mu=\int {\mu}_x d\mu(x),$ be the ergodic decomposition of $\mu$ . Then, if for $\mu$-a.e\ $x$ the systems $(X,{\mu}_x,u)$, $(Y,{\nu},u)$ are disjoint, then $(X,\mu,u)$, $(Y,{\nu},u)$ are disjoint as well.
\end{lemma} 
We give both proofs for the sake of completeness.
\begin{proof}[Proof of Lemma~\ref{disjoint}]
Let $\eta$ be a joining of $\mu,{\nu}$. As $\eta$ is a joining, the projections $\on{p}_X:(X\times Y,\eta)\to (X,\mu),\; \on{p}_Y:(X\times Y,\eta)\to (Y,{\nu})$ induce injections of the $L^2$-spaces; that is, the Hilbert spaces $\cH_X\defi L^2(X,\mu),\cH_Y\defi L^2(Y,{\nu})$ are isometrically embedded in $\cH\defi L^2(X\times Y,\eta)$. We denote by $\cH^0,\cH_X^0,\cH_Y^0$ the subspaces orthogonal to the constant functions in each of the spaces. Showing that $\eta$ is the product measure is the same as showing 
that the two subspaces $\cH_X^0,\cH_Y^0$ of $\cH^0$ are orthogonal. For this, note that the mixing assumption for the $u$-action
on $(Y,{\nu})$ is equivalent to saying that for each $v\in \cH_Y^0$ the sequence $T_u^nv$ converges weakly to 0, where we write $T_u$ for the unitary operator on $\cH$ induced by the $u$-action. If the two subspaces are not orthogonal,
then there exists a vector $0\ne v\in \cH_Y^0$ that has a non-trivial projection  to $\cH_X^0$. Let us denote this projection of $v$ by $w\in\cH_X^0$. We claim that $T_u^n w$ must converge weakly to 0 as well. This simply follows from the fact that the projection commutes with $T_u$ as the subspace we project on is $T_u$-invariant and $T_u$ is unitary. However, as $(X,\mu,u)$ is isomorphic to a minimal rotation on a compact group
there exists an integer sequence~$n_k\to\infty$ such that~$T_u^{n_k}w\to w$ as~$k\to\infty$. This contradiction implies the lemma.
\end{proof}
\begin{proof}[Proof of Lemma~\ref{e.c.r}]
Let $\eta$ be a joining of $\mu,{\nu}$ and let
\[
 \eta=\int \eta_{(x,y)} d\eta(x,y)
\]
be the ergodic decomposition of $\eta$. Let $\on{p}_X,\on{p}_Y$ denote the projections from $X\times Y$ to $X,Y$ respectively. For $\eta$-a.e\ $(x,y)$ we have that $(\on{p}_Y)_*\eta_{(x,y)}$ is an ergodic $u$-invariant probability measure on $Y$ and 
as $\eta$ projects to ${\nu}$ we deduce that ${\nu}=\int (\on{p}_Y)_*\eta_{(x,y)} d\eta(x,y)$. Since ${\nu}$ is ergodic we conclude that ${\nu}=(\on{p}_Y)_*\eta_{(x,y)}$ for 
$\eta$-almost any $(x,y)$. By a similar reasoning one can argue that for $\eta$-a.e\ $(x,y)$ $(\on{p}_X)_*\eta_{(x,y)}={\mu}_x$. It follows that for $\eta$-almost any $(x,y)$ the ergodic component $\eta_{(x,y)}$ is a joining of ${\mu}_x$ and ${\nu}$ and therefore by our disjointness assumption we have $\eta_{(x,y)}={\mu}_x\times {\nu}$ for $\eta$-almost any pair $(x,y)$.  This implies the lemma.

 \end{proof}

\subsection{Measures supported in $Ux_0\times Vx_0$}
We will use the following notation in the special case where $m=n=2$. 
Let us denote by $\mb{u}\mapsto \mb{u}^\#$ the linear isomorphism of $\on{Mat}_2(\bR)$ given by
$\smallmat{a&b\\c&d}^\#=\smallmat{d&-c\\-b&a}.$
For co-prime integers $p,q$ we denote 
\begin{equation}\label{Wpq}
L_{p,q}\defi\set{\pa{\smallmat{I&0\\p\mb{u}&I},\smallmat{I&q\mb{u}^\#\\0&I}}:\mb{u}\in\on{Mat}_2(\bR)}<U\times V.
\end{equation}
\begin{theorem}\label{mcn=m}
Assume $n=m$. Let $\eta$ be a $\Lam_\Del$-invariant and ergodic probability measure supported on $Ux_0\times Vx_0$ and suppose that $(\pi_1)_*\eta=\mb{m}_{Ux_0}$, 
$(\pi_2)_*\eta=\mb{m}_{Vx_0}$. Then 
$\eta=\mb{m}_{Ux_0}\times \mb{m}_{Vx_0}$,
or $n=2$ and $\eta=\frac{1}{N}\sum_1^N\mb{m}_{L_{p,q}(x_i,y_i)}$, where $p,q$ are co-prime integers
and the points $x_i,y_i$ may be chosen to have rational matrix representatives. 
\end{theorem}
\begin{proof}
We abuse notation and write $\Lam$ for $\Lam_\Del$.
We identify 
$Ux_0\times Vx_0$ with the product torus $\bT^{n^2}\times \bT^{n^2}$ and recall that the 
$\Lam=\SL_n(\bZ)\times \SL_n(\bZ)$-action on it is  
induced by the $\Lam$-representation on $\Mat_n(\bR)\times \Mat_n(\bR)$ given by (see~\eqref{action1}, \eqref{action2}),
\begin{equation}\label{r.f}
(\del_1,\del_2)(\bu,\mb{v})=(\del_2\bu \del_1^{-1}, \del_1\mb{v}\del_2^{-1}).
\end{equation}

As $\Lam$ is generated by unipotents, a suitable application of Theorem~\ref{NimishRatner1} implies that there exists a $\Lam$-invariant subspace $W\subset \Mat_n(\bR)\times \Mat_n(\bR)$  and finitely many periodic orbits $\set{W(\bu_i,\mb{v}_i)}_1^N\subset  \bT^{n^2}\times \bT^{n^2}$ such that $\eta= \frac{1}{N}\sum_1^N\mb{m}_{W(\bu_i,\mb{v}_i)}$.
From our assumption that $\eta$ is a joining of $\mb{m}_{Ux_0}$ and $\mb{m}_{Vx_0}$ we conclude that $W$ projects onto $\Mat_n(\bR)$ under
both left and right projections.

We will need the following general representation theoretic lemma. Its proof is straightforward and left to the reader. 
\begin{lemma}\label{r.t.l}
Let $\rho_i$ for $i=1,2$ be two irreducible representations of a group $\Lam$ on the vector spaces $V_i$ and let $\rho=\rho_1\oplus \rho_2$. If $\set{0}\varsubsetneq W\varsubsetneq V_1\oplus V_2$
is a $\Lam$-invariant subspace, then either $W=V_1\times \set{0}$, or $W=\set{0}\times V_2$, or there is an isomorphism $\vphi:(V_1,\rho_1)\to( V_2,\rho_2)$ of $\Lam$-representations such that $W=\set{(v,\vphi(v)):v\in V_1}$.
\end{lemma}
We apply this Lemma with $V_1=V_2=\Mat_n(\bR)$, with the representations $\rho_i$ of $\Lam$ which are given by restricting to left and right  coordinates of the formula~\eqref{r.f}, and with $W$ being the subspace identified above which is attached to $\eta$. If $W$ is the whole space then $\eta=\mb{m}_{Ux_0}\times \mb{m}_{Vx_0}$. Otherwise, we deduce from the lemma 
(and the fact that $W$ projects onto each factor), the existence of the isomorphism $\vphi$. Then, applying~\eqref{r.f} to the diagonal copy of $\SL_n(\bZ)$ in $\Lam$, we deduce that 
the image of the line $\set{sI:s\in\bR}\subset V_1$ must be stable under conjugation by every element
of $\SL_n(\bZ)$, which implies that this line must be mapped to itself. Therefore, there exists a scalar $\rho_\vphi$ such that
 $\vphi(I)=\rho_\vphi I$. It then follows from~\eqref{r.f} that the restriction of $\vphi$ to the set 
 $\set{s\del: s\in\bR,\del\in\SL_n(\bZ)}$ is given by the formula  
\begin{equation}\label{another formula}
\vphi(s\del)=\rho_\vphi s\del^{-1}.
\end{equation}
 For $n>2$ this is not a linear map and so the existence of $\vphi$ is ruled out and so indeed $\eta=\mb{m}_{Ux_0}\times \mb{m}_{Vx_0}$ as claimed.

In the case $n=2$ on the other hand, formula~\eqref{another formula}
is given by 
\begin{equation}\label{isodim2}
\mb{u}=\smallmat{a&b\\ c& d}\overset{\vphi}{\lra}\rho_\vphi\smallmat{d&-b\\-c&a}=\rho_\vphi \mb{u}^\#,
\end{equation}
which is linear, and by Zariski density, this must be the formula for $\vphi$ on  $V_1=\on{Mat}_2(\bR)$. 
By the above lemma we have that $$W=\set{\pa{\mb{u},\rho_\vphi\mb{u}^\#}:\mb{u}\in\Mat_n(\bR)}$$ and since
it has periodic orbits in $\bT^{n^2}\times \bT^{n^2}$, $\rho_\vphi$ must be rational, say $\rho_\vphi=p/q$ for
some co-prime integers $p,q$. Under our identifications $W$ corresponds to to the subgroup $L_{p,q}$ from~\eqref{Wpq}.

Finally, we need to justify the rationality of the points $(\mb{u}_i,\mb{v}_i)$. Consider the action of $\Lam$ on the quotient torus
$\bT^{n^2}\times \bT^{n^2}/W$ which we identify with the standard torus $\bT^N$ for a suitable $N$.
The measure $\eta$ projects there to a finitely supported $\Lam$-invariant measure. It is straightforward to show the existence of
$\lam\in\Lam$ which acts on $\bR^N\cong \Mat_n\times\Mat_n/W$ without roots of unity as eigenvalues and therefore by Lemma~\ref{rationalitylemma0} we conclude the rationality of the images of $(\mb{u}_i,\mb{v}_i)$ in $\bT^N$ and in turn the desired
rationality before projecting (because $W$ is a rational space).
\end{proof}

\section{Hecke friends}\label{seced}
\subsection{Hecke friends}\label{hf} For any positive integer $\ell$ and an $m$-dimensional lattice $x\in X_m$ we define the set of \textit{$\ell$-Hecke friends} of $x$ to be the collection of 
lattices $$\set{\ell^{-\frac{1}{m}}\Del:\textrm{ $\Del$ is a subgroup  of index $\ell$ of $x$}}\subset X_m.$$ 
Clearly, it is enough to understand the collection of $\ell$-Hecke friends of $\bZ^m$, for if $x=h\bZ^m$ for some $h\in\SL_m(\bR)$, then the collection of $\ell$-Hecke friends of $x$ is simply the image under $h$ of the corresponding collection for $\bZ^m$. 

Given a subgroup $\Del<\bZ^m$ of index $\ell$, we choose $\mb{u}\in\Mat_m(\bZ)$ 
whose columns form a basis for $\Del$ and with $\det\mb{u}>0$. By
Lemma~\ref{r.f.lemma} (and its proof) there is a unique positive elementary divisors tuple 
$(\ell_1\dots\ell_m)$ attached to it. Consequently, $\ell=\prod_1^m\ell_i$. It is straightforward to show that the tuple does not depend on 
the choice of $\mb{u}$ and so we make the following. 
%
\begin{definition}\label{def ht}
Given an $\ell$-Hecke friend of $\bZ^m$, we define
its \textit{Hecke-type} to be the corresponding $m$-tuple of elementary divisors. 
\end{definition}
The above discussion could be restated as follows.
\begin{lemma}\label{hfdsorbits}
 The collection of $\ell$-Hecke friends of $\bZ^m$ is partitioned into $\SL_m(\bZ)$-orbits  in the following way
$$\set{\ell\textrm{-Hecke friends of }\bZ^m}=\bigsqcup_{(\ell_1\dots\ell_m)}\SL_m(\bZ)\diag{\ell_1,\dots,\ell_m}\bZ^m,$$
where the union is taken over the $\ell$-Hecke-types; that is, all positive tuples $(\ell_1\dots \ell_m)$ such that $\ell_i|\ell_{i+1}$ and $\prod_1^m\ell_i=\ell$.  
\end{lemma}

The following equidistribution result will be needed in the proof of Theorem~\ref{JED}.
\begin{theorem}\label{edhf}
Let $\set{(\ell_{i1}\dots\ell_{im})}_{i=1}^\infty$ be a sequence of types of $\bm{\ell}_i$-Hecke friends of $\bZ^m$, where $\bm{\ell}_i\defi\prod_{j=1}^m\ell_{ij}$, and let
${\nu}_i$ denote the normalized counting measure  supported on the collection of $\bm{\ell}_i$-Hecke friends of $\bZ^m$ of type $(\ell_{i1}\dots\ell_{im})$. Then ${\nu}_i\wstar \mb{m}_{X_m}$ if and only if $\frac{\ell_{im}}{\ell_{i1}}\lra\infty$.
\end{theorem}
\begin{proof}
This theorem follows from~\cite{ClozelOhUlmo}. More conveniently, it is a special case of 
\cite[Theorem 1.2]{EskinOh-Hecke}.
This result states that for $q_i\in\GL_m(\bQ)$, the normalized counting measure on the (finite) orbit 
$\frac{1}{(\det q_i)^{1/m}}\SL_m(\bZ)q\bZ^m$ equidistribute to $\mb{m}_{X_m}$ once $\deg q_i\to\infty$
where $\deg q_i$ is the size of the corresponding orbit. We apply this for
$q_i=\diag{\ell_{i1},\dots,\ell_{im}}$. We need to explain why the condition
$\deg(\diag{\ell_{i1}\dots\ell_{im} })\to\infty$
is equivalent to the requirement $\frac{\ell_{im}}{\ell_{i1}}\lra\infty$. 

Without loss of generality we may assume on scaling by $1/\ell_{i1}$ that $\ell_{i1}=1$. 
If $\ell_{im}$ is bounded along some subsequence
of $i$'s then clearly the degrees do not diverge to $\infty$. 
On the other hand, if $(1,\dots,\ell_m)$ is a Hecke type and $\ell_m>M$, then there exists
some $1\le r<m$ such that $\frac{\ell_{r+1}}{\ell_r}>M^{1/m}$. Then the orbit $\SL_m(\bZ)\diag{1,\dots,\ell_m}\bZ^m$
contains an embedded copy of the orbit $\SL_2(\bZ)\smallmat{1&0\\0&\ell_{r+1}/\ell_r}\bZ^2$ which contains at least $p^{e-1}(p+1)$ 
points where $p^e$ is the largest
prime power dividing $\ell_{r+1}/\ell_r$. As this prime power goes to $\infty$ with $M$, we are done.
\end{proof}

\section{Proof of Theorem~\ref{JED}}\label{proof of jed} 

We follow the scheme presented after ~\eqref{diagram 1}, divide the proof into Steps 1-4.
\subsection{Step 1}\label{step1section} In this step we establish the convergence $\bm{\mu}_k\wstar \mb{m}_{Ux_0}$. Let $\sig$ be a weak$^*$ limit of the sequence $\set{\bm{\mu}_k}$. By Lemma~\ref{invariance under Lambda}\eqref{invlem1} $\sig$ is a $\Lam$-invariant probability measure. 
By Lemma~\ref{s.classification} there are only countably many $\Lam$-invariant ergodic probability measures on 
$Ux_0$.  We let $\sig_0\defi \mb{m}_{Ux_0}$ and let $\set{\sig_i}_{i=1}^\infty$ be any enumeration of the 
measures supported on finite $\Lam$-orbits. 
By the ergodic decomposition we may write $\sig=\sum_{i=0}^\infty c_i\sig_i$ where $c_i\ge 0$ and $\sum_i c_i=1$. 
The proof of Step 1 will be concluded once we show that $c_0=1$. By Lemma~\ref{s.classification} each 
$\sig_i$, $i>0$ is supported in the torsion points $\on{Tor}_{k_i}(Ux_0)$ for some integer $k_i$. It is straightforward to show that for $k>k_i$ we have that 
$\cP_k\cap\on{Tor}_{k_i}(Ux_0)=\varnothing$, and so for all large enough $k$ the support of $\bm{\mu}_k$ is disjoint from the orbit 
on which $\sig_i$ is supported. We now apply the Non-accumulation Theorem~\ref{n.a.t} and deduce that $c_i=0$ for~$i>0$ as desired. The application of Theorem~\ref{n.a.t} 
is done with the following choices: $G=\Lam\ltimes U$, $\Ga=\Lam\ltimes \on{Mat}_{n\times m}(\bZ)$, $L=\Lam$, $z$ being one of the points in the support of $\sig_i$. The space $W$, 
being a linear complement of $\lie(L)=\set{0}$,
equals $\lie(U)= \on{Mat}_{n\times m}(\bR)$ and indeed there are no $\Lam$-fixed vectors in $W$ other than zero.
\subsection{Step 2}
In this step we establish the convergence $\bm{\nu}_k\wstar (\pi_3)_*\bm{\theta}$. This follows from the following lemma.
\begin{lemma}\label{link to Hecke friends}
The measure $\bm{\nu}_k\defi (\pi_3)_*a(k)_*\bm{\mu}_k$ is the normalized counting measure on the collection of 
$k^n$-Hecke friends of  $\bZ^m$ of type $(\underbrace{1\dots 1}_{m-n} \nolinebreak[2] \underbrace{k\dots k}_{n}).$
In particular, if $m>n$, $\bm{\nu}_k$ equidistribute to $\mb{m}_{X_m}$ and if $m=n$ then $\bm{\nu}_k$ is the Dirac measure supported on $x_0$.
\end{lemma}
\begin{proof}
The last sentence in the statement simply follows from Theorem~\ref{edhf} (and the observation that $\bZ^m$ is the only $k^n$-Hecke friend of type $(k\dots k)$ of itself). We are thus left to verify the first sentence of the statement.

The set $\cP_k$ decomposes into $\Lam$-orbits and each such orbit is mapped by $\pi_3\textrm{{\tiny $\circ$}} a_k$ to a single $
\Lam$-orbit in $X_m$. Let $\Om\subset \cP_k$ be such an orbit. By Lemma~\ref{r.f.lemma} there exists $x_{k^{-1}\mb{u}}\in\Om$ with 
$\mb{u}$ of the form~\eqref{reduced form}. By Lemma~\ref{hfdsorbits}, if we will show that $\pi_3(a_k x_{k^{-1}\mb{u}})$ is a Hecke friend of $\bZ^m$ of
the type prescribed in the statement, it would follow that the normalized counting measure on $\Om$ is pushed by 
$(\pi_3)_*(a_k)_*$ to the normalized 
counting measure on the set of Hecke friends on $\bZ^m$ of this type\footnote{Note that from here
follows a stronger statement than the one sought in Step~2 which is the reason 
for the strengthening of Theorem~\ref{JED} discussed in~\S\ref{c.r}.}. This implies the same statement regarding $\bm{\nu}_k$.

To this end, let $x_{k^{-1}\mb{u}}\in \cP_k$ with $\bu$ $k$-primitive of the form~\eqref{reduced form}. We now calculate
$\pi_3(a(k)x_{k^{-1}\bu})$. Let $\ga=\smallmat{A_\ga& B_\ga\\C_\ga& D_\ga}$ be a matrix solving~\eqref{basic eq} so that
\begin{equation}\label{eq.s.b}
\pi_3(a(k)x_{k^{-1}\bu})=k^{-\frac{n}{m}}A_\ga\bZ^m.
\end{equation} 
It follows from the proofs of Lemmas~\ref{basic lemma}, \ref{lemprimc} that a possible choice for $\ga$ is the inverse of the matrix $\del$ appearing in~\eqref{bigmatrix}. A short computation shows that (using the formula for the inverse of a $2\times 2$ matrix) this inverse is
\begin{equation}\label{big matrix2}
\del^{-1}=
\mat{
\diag{k\dots k}& 0 & -\diag{f_1\dots f_n}\\
0 &I_{n-m}&0\\
-\diag{\ell_1\dots\ell_n}&0&\diag{e_1\dots e_n} 
}.
\end{equation}
In other words $\pi_3(a(k)x_{k^{-1}\bu})$ is the $m$-dimensional unimodular lattice 
$$k^{-\frac{n}{m}}\diag{\underbrace{k,\dots,k}_{n},\underbrace{1,\dots,1}_{m-n}}\bZ^m$$ which is a Hecke-friend of $\bZ^m$
of the desired type.
\end{proof} 
\subsection{Step 3}
In this step we establish the convergence $a(k)_*\bm{\mu}_k\wstar \bm{\theta}$. 
Before turning to the proof we need notation and a lemma.
The orbit $Hx_0$ breaks into periodic $V$-orbits which are the $\pi_3$-fibers. In
order to state the next lemma we need to define the notions of $k$-torsion and $k$-primitive points in such a fiber.
For $x\in Hx_0$, we write $x=\smallmat{g&\mb{v}\\0&I}x_0$ and set
\begin{equation}\label{TorV}
\on{Tor}_k(Vx)\defi\set{\smallmat{g&0\\0&I}\smallmat{I&k^{-1}\mb{v}\\0&I}x_0\in Vx: \mb{v}\in\Mat_{m\times n}(\bZ)}.
\end{equation}
\begin{equation}\label{TorVp}
\on{Tor}_k^{\on{prim}}(Vx)\defi\set{\smallmat{g&0\\0&I}\smallmat{I&k^{-1}\mb{v}\\0&I}x_0\in Vx: \mb{v}^t\textrm{ is $k$-primitive}}.
\end{equation}
Note that although $g$ is not well defined, the coset $g\SL_m(\bZ)$ is well defined and therefore $\on{Tor}_k(Vx), \on{Tor}_k^{\on{prim}}(Vx)$ are well defined
as well.
\begin{lemma}\label{lemtorprim}
For any positive integer $k$, 
\begin{equation}\label{eqinc}
a_k\cP_k \subset \bigcup_{x\in X_m} \on{Tor}_k^{\on{prim}}(\pi_3^{-1}(x)).
\end{equation}
\end{lemma}
\begin{proof}
We first observe that the sets on both sides of~\eqref{eqinc} are $\Lam$-invariant. Regarding $a_k\cP_k$ this follows from
Lemma~\ref{invariance under Lambda}\eqref{invlem2}. Regarding the set on the right, 
given $\lam=\smallmat{\del_1&0\\0&\del_2}\in\Lam$
and a point $\smallmat{g&0\\0&I}\smallmat{I &k^{-1}\mb{v}\\0&I}x_0\in \on{Tor}_k^{\on{prim}}(\pi_3^{-1}(x))$,
where $x=\smallmat{g&0\\0&I}x_0$ and $\mb{v}^t$ $k$-primitive, then
$$\lam \smallmat{g&0\\0&I}\smallmat{I &k^{-1}\mb{v}\\0&I}x_0=\smallmat{\del_1g\del_1^{-1}&0\\0&I}\smallmat{I &k^{-1}\del_1\mb{v}\del_2^{-1}\\0&I}x_0,$$
which belongs to $\on{Tor}_k^{\on{prim}}(\pi_3^{-1}(\lam x))$ because $(\del_1\mb{v}\del_2^{-1})^t$ is $k$-primitive.

It now follows from Lemma~\ref{r.f.lemma} that it is enough to show that $a_k x_{k^{-1}\mb{u}}$ is in the right hand side of~\eqref{eqinc} for 
$\mb{u}$ as in~\eqref{reduced form}. For such a choice of $\bu$ we have seen in the course of the proof of Lemma~\ref{link to Hecke friends} that the matrix 
$\ga$ solving equation~\eqref{basic eq} may be chosen to be as in~\eqref{big matrix2}. Thus~\eqref{basic eq} now becomes
\begin{equation}\label{eq113}
a_k x_{k^{-1}\mb{u}}=\smallmat{ k^{-m/n}A_\ga &k^{-m/n}B_\ga\\0&I}x_0=\smallmat{k^{-m/n}A_\ga&0\\0&I}\smallmat{I&A_\ga^{-1}B_\ga\\0&I}x_0,
\end{equation} 
and indeed the concrete
form of $\ga$~\eqref{big matrix2} gives by a straightforward calculation that 
\begin{equation}\label{eq114}
kA_\ga^{-1}B_\ga=\smallmat{-\diag{f_1\dots f_n}\\0_{m-n\times m-n}}, 
\end{equation}
so indeed $(kA_\ga^{-1}B_\ga)^t$ is $k$-primitive.
\end{proof}
In particular, we record the following corollary for future reference.
\begin{corollary}\label{r.n=m}
For any $k>k_0$ and any $x\in X_m$, $\on{supp}\pa{a(k)_*\bm{\mu}_k}\cap \on{Tor}_{k_0}(\pi_3^{-1}(x))=\varnothing$.

\end{corollary}
We now turn to the proof of Step 3.
Let $\sig$ be a weak$^*$ accumulation point of $\set{a(k)_*\bm{\mu}_k}_{k=1}^\infty$. By Lemmas~\ref{basic lemma}, \ref{invariance under Lambda},
$\sig$ is a $\Lam$-invariant measure supported in $Hx_{0}$.
By Step 2 we conclude that $\sig$ is a probability measure that projects under $\pi_3$ to either $\mb{m}_{X_m}$ or $\del_{x_0}$ according to whether $m>n$, $m=n$ respectively.
By Corollary~\ref{summarizing}, there are only countably many ergodic $\Lam$-invariant probability measures supported in $Hx_{0}$ and so let $\set{\sig_i}_{i=0}^\infty$ be an enumeration of them. By the ergodic decomposition we may represent $\sig=\sum_{i=0}^\infty c_i\sig_i$, where $c_i\ge 0$ and $\sum_{i=0}^\infty c_i=1$. 

 \underline{Assume $m=n$}. 
 Let us denote  $\sig_0=\mb{m}_{Vx_0}$. With 
 this notation we  will conclude this case once we show that $c_i=0$ for $i>0$. 
Fixing $i>0$ and using the terminology of Definition~\ref{t.o.m}, $\sig_i$ is of type~\eqref{c.g.1}--\eqref{c.g.4}. According to the type we will show that
$c_i=0$.
 
 It is impossible that $\sig_i$ is of type~\eqref{c.g.2} because since by Step 2  $(\pi_3)_*\sig=\del_{x_0}$, it would follow
 from Corollary~\ref{summarizing}\eqref{sum2} that $\sig_i=\mb{m}_{Vx_0}=\sig_0$.
Similarly, Corollary~\ref{summarizing} implies that if $\sig_i$ is of type~\eqref{c.g.3} or~\eqref{c.g.4} then $c_i=0$. 
 Finally,
 if $i$ is such that $\sig_i$ is of type~\eqref{c.g.1} then by Corollary~\ref{r.n=m} we have that the finite $\Lam$-orbit on which $\sig_i$ is supported on
 is disjoint, for large values of $k$, from the support of $a(k)_*\bm{\mu}_k$. By the non-accumulation Theorem~\ref{n.a.t} we deduce that $c_i=0$. 
 Note that, similarly to the discussion at the end of \S\ref{step1section}, the application of Theorem \ref{n.a.t} was done with the following choices:
 $G=\Lam\ltimes V, \Ga=\Lam\ltimes\on{Mat}_{m\times n}(\bZ)$, $L=\Lam$, $z$ being a point in the support of $\sig_i$ and 
 $W=\lie(V)=\on{Mat}_{m\times n}(\bR)$ which indeed has no nonzero $\Lam$-fixed vectors.
This 
 concludes the proof of Step 3 for the case $n=m$.
 
 \underline{Assume $m>n$}.  
 Let us denote  $\sig_0=\mb{m}_{Hx_{0}}$ -- the only measure of type~\eqref{c.g.4}. With 
 this notation we will conclude this case once we show that $c_i=0$ for all $i>0$. We do this by considering the possible types.
By Step 2 $(\pi_3)_*\sig=\mb{m}_{X_m}$ so we conclude from Corollary~\ref{summarizing}  that for any $i$ such that  $\sig_i$ is of type~\eqref{c.g.1} or~\eqref{c.g.2} we have that $c_i=0$. Let $i>0$ be such that $\sig_i$ is of type~\eqref{c.g.3}. 
Note that by Corollary~\ref{summarizing}, the intersection of the support of a measure of type~\eqref{c.g.3} with any $\pi_3$-fiber $\pi_3^{-1}(x)$ is contained in $\on{Tor}_{k_0}(\pi_3^{-1}(x))$ for a fixed $k_0$
that does not depend on the choice of $x\in X_m$.
 By Corollary~\ref{r.n=m} we conclude that for $k>k_0$
 the support of $a(k)_*\bm{\mu}_k$ is disjoint from the periodic $\Lam\cdot \SL_m(\bR)$-orbit on which $\sig_i$ is supported. 
 It now follows from the Non-accumulation Theorem~\ref{n.a.t} that $c_i=0$. 
 The application of Theorem~\ref{n.a.t} is done with the following choices: $G=\Lam\cdot H$, $\Gamma=\Lam\ltimes\on{Mat}_{m\times n}(\bZ)$, $L=\Lam \cdot\SL_m(\bR)$,
 $z$ any point in the support of $\sig_i$, and $W=\lie(V)=\on{Mat}_{m\times n}(\bR)$.
 This concludes the proof of Step 3.
 \subsection{Step 4}
In this step we conclude the proof of Theorem~\ref{JED}; that is, we establish the convergence 
$\wt{a}(k)_*\wt{\bm{\mu}}_k\wstar \mb{m}_{Ux_0}\times \bm{\theta}$. 
 At some point close to the end of the proof we will need the following lemma\footnote{This lemma is needed only in the case $m=n=2$.}.
\begin{lemma}\label{r.n=m0}
In the case\footnote{Compare this with~\eqref{set3}.} $m=n$, 
$$\wt{a}_k\wt{\cP}_k=\set{\pa{\smallmat{I&0\\ k^{-1}\mb{u}&I}x_0, \smallmat{I&k^{-1}\bu^*\\0&I}x_0}:\mb{u}\textrm{ is $k$-primitive}},$$
where $\bu^*\in\Mat_n(\bZ)$ is an inverse of $\bu$ modulo $k$. 
\end{lemma}
\begin{proof}
The proof is a further inspection of the arguments in Lemmas~\ref{link to Hecke friends}, \ref{lemtorprim} so we keep it terse. 
We need to show that $a_kx_{k^{-1}\mb{u}}=\smallmat{I& k^{-1}\mb{u}^*\\0&I}x_0$ and by acting with $\Lam$ we see that it is enough to show this for
$\mb{u}$ as in~\eqref{reduced form}. For such $\mb{u}$ this follows from~\eqref{eq113}, \eqref{eq114}, taking into account the concrete form
of $\ga$ given by~\eqref{big matrix2}.
\end{proof}

We now turn to the proof of Step 4. Let $\wt{\sig}$ be a 
weak$^*$ accumulation point of the sequence $\set{\wt{a}(k)_*\wt{\bm{\mu}}_k}$. 
By Steps 1 and 3 we know that $(\pi_1)_*\wt{a}(k)_*\wt{\bm{\mu}}_k\wstar \mb{m}_{Ux_0}$, 
$(\pi_2)_*\wt{a}(k)_*\wt{\bm{\mu}}_k\wstar \bm{\theta}$. As $\pi_2$ is proper we deduce from the fact 
that $\bm{\theta}$ is a probability measure, that $\wt{\sig}$ is a probability measure. This in turn implies that 
$(\pi_1)_*\wt{\sig}=\mb{m}_{Ux_0}$, $(\pi_2)_*\wt{\sig}=\bm{\theta}$.
Also, by Lemma~\ref{invariance under Lambda}\eqref{invlem5}, $\wt{\sig}$ is $\Lam_\Del$-invariant. 
Hence, the dynamical system $(Ux_0\times Hx_{0},\wt{\sig},\Lam_\Del)$ is a joining of $(Ux_0,\mb{m}_{Ux_0},\Lam)$, $(Hx_{0},\bm{\theta},\Lam)$. 
By Theorems~\ref{mcn<m}, \ref{mcn=m}, we deduce that if $(m,n)\ne(2,2)$ then $\wt{\sig}=\mb{m}_{Ux_0}\times \bm{\theta}$ and the proof is concluded. 
We are thus left to deal with the case $n=m=2$. 
\subsubsection{The case $n=m=2$}\label{two two case}
As the projections $\mb{m}_{Ux_0},\mb{m}_{Vx_0}$ are $\Lam$-ergodic we conclude that (almost) any ergodic component of $\wt{\sig}$ is a joining 
of $\mb{m}_{Ux_0},\mb{m}_{Vx_0}$ as well. By Theorem~\ref{mcn=m} there are only countably many such ergodic components and we will be done once we 
prove the following lemma in which we use the notation of~\eqref{Wpq}.
\begin{lemma}
Let $p,q$ be  co-prime integers and let $(x,y)\in Ux_0\times Vx_0$. Then, $\wt{\sig}(L_{p,q}(x,y))=0$.
\end{lemma}
\begin{proof} 
By Corollary~\ref{r.negligible} we deduce that it is enough to prove that 
\begin{equation}\label{neg pf}
\frac{|\wt{a}_k\wt{\cP}_k\cap L_{p,q}(x,y)|}{|\wt{\cP}_k|}\lra 0\textrm{ as }k\to\infty.
\end{equation}
Here the application of Corollary~\ref{r.negligible} is done with the following choices: $G=\Lam_\Del\ltimes (U\times V)$, 
$\Ga=\Lam_\Del\ltimes (U(\bZ)\times V(\bZ))$.
$L=L_{p,q}$ and $z=(x,y)$. The existence of the $\Lam_\Del$-invariant linear complement of $\lie(L_{p,q})$ follows from the semi-simplicity of 
the group $\SL_2(\bR)\times\SL_2(\bR)$ and the fact that $\lie(L_{p,q})$ is an invariant subspace of $\on{Mat}_2(\bR)\times\on{Mat}_2(\bR)$
under the representation~\eqref{r.f}. 

Let us denote for $\mb{v}\in\on{Mat}_2(\bR)$, $y_{\mb{v}}\defi\smallmat{I&\mb{v}\\0&I}x_0\in Vx_0$,
and for $\mb{u}\in \GL_2(\bZ/k\bZ)$ we write (using the notation of Lemma~\ref{r.n=m0}),
\begin{equation}\label{Phi}
\Phi(\bu)\defi(x_{k^{-1}\bu},y_{k^{-1}\bu^*}).
\end{equation} 
Lemma~\ref{r.n=m0} gives that $\Phi$ is a bijection between $\GL_2(\bZ/k\bZ)$ and
\begin{equation}\label{eq corresp}
\wt{a}_k\wt{\cP}_k=\set{\Phi(\bu):\bu\in\GL_2(\bZ/k\bZ)}.
\end{equation} 
The left and right actions of $\GL_2(\bZ/k\bZ)$ on itself induce (via $\Phi$) actions of it on $\wt{\cP}_k$. 
Fixing $\bu=\smallmat{a&b\\c&d}\in\GL_2(\bZ/k\bZ)$ we will establish~\eqref{neg pf} by
analyzing the orbit $\Phi\pa{\set{g_\ell\bu:\ell\in(\bZ/k\bZ)^\times}}$, 
where $g_\ell\defi\smallmat{\ell&0\\0&1}$, which is of size $\phi(k)$ (here $\phi$ is the Euler function), and proving that 
\begin{equation}\label{new goal}
\frac{\av{\set{\ell :\Phi(g_\ell\bu)\textrm{ is in the same $L_{p,q}$-orbit of }\Phi(\bu)}}}{\phi(k)}\overset{k\to\infty}{\lra}0.
\end{equation}
Following the definition we see that $\Phi(\bu),\Phi(g_\ell\bu)$ are in the same $L_{p,q}$-orbit if and only if  the difference
$(g_\ell\bu,\bu^*g_\ell^*)-(\bu,\bu^*)$ (thought of as an element of $\on{Mat}_2(\bR)\times \on{Mat}_2(\bR)$)
is in $\lie(L_{p,q}) + \on{Mat}_2(k\bZ)\times\on{Mat}_2(k\bZ)$. 
In other words, if we let $j=\det \bu$ and recall that $\bu^*=j^*\smallmat{d&-b\\-c&a}$, $g_\ell^*=\ell^*\smallmat{1&0\\0&\ell}$, 
and the definition of $L_{p,q}$ in~\eqref{Wpq}, we see that $\Phi(\bu),\Phi(g_\ell\bu)$  are in the same $L_{p,q}$-orbit if and only if
\begin{equation}\label{conc eq}
p \smallmat{0&(1-\ell)b\\0&(\ell-1)a}=q\smallmat{(\ell^*-1)j^*d&0\\(1-\ell^*)j^*c&0}\;\on{mod} k\bZ.
\end{equation}
We claim that \eqref{conc eq} implies $\ell =1 \on{mod} k\bZ$. For that purpose we split $k=k_1k_2$ with $k_1$ coprime to $p$ and $k_2$ coprime to $q$. Using the second column in \eqref{conc eq} we get $(\ell-1)b\in k_1\bZ$ and $(\ell-1)a\in k_1\bZ$. However, since $a,b$ form the first row of an invertible matrix modulo $k$ we see that the greatest common divisor of $a,b,k_1$ is one and we conclude $\ell=1\on{mod}k_1\bZ$. Using the first column of \eqref{conc eq} and invertibility of $j^*$ modulo $k_2$ in the same way we get $\ell^*=1\on{mod}k_2\bZ$. Together this shows $\ell=1\on{mod}k\bZ$ and so \eqref{new goal} (since $\phi(k)\to\infty$). 
This concludes the proof the lemma and by that the proof of Theorem~\ref{JED}.
\end{proof}

\section{A strengthening of Theorem~\ref{JED}}\label{c.r}
Before we turn to the closing section of this paper we observe that the proof presented above of Theorem~\ref{JED} actually establishes more than stated. 
So far we looked at (various images of) the counting measure on $\wt{\cP}_k$ and used its $\Lam_\Del$-invariance but $\wt{\cP}_k$ splits into $\Lam_\Del$-orbits
and we could restrict our attention to single orbits. 
Let 
\begin{equation}\label{Ik}
I_k\defi\set{\vec{\ell}=(\ell_1,\dots,\ell_n)\in\bZ^n:\ell_i|\ell_{i+1}, \gcd(\ell_i,k)=1},
\end{equation}
and denote by $\br{I_k}$ the image of $I_k$  in $(\bZ/k\bZ)^n$. We refer to the elements of $\br{I_k}$ as \textit{$k$-primitive elementary divisors tuples} and sometimes abuse notation and do not distinguish between elements in 
$\br{I_k}$ and their integer representatives.
It follows from Lemma~\ref{r.f.lemma} that a $\Lam_\Del$-orbit in $\wt{\cP}_k$  is a set of the form
\begin{multline*}
\cP_{k,\vec{\ell}}\defi \{(x_{k^{-1}\mb{u}},x_{k^{-1}\mb{u}}):{\tiny \textrm{the elementary divisors tuple of $\mb{u}$ equals $\vec{\ell}\on{mod} k$}}\}.
\end{multline*} 
Let $\wt{\bm{\mu}}_{k,\vec{\ell}}$ be the normalized counting measure on $\wt{\cP}_{k,\vec{\ell}}$ so that $\wt{\bm{\mu}}_k$ is an average of the $\wt{\bm{\mu}}_{k,\vec{\ell}}$'s.
\begin{theorem}\label{strong JED}
If  $(m,n)\ne (2,2), (1,1)$. Then, for any choice of a sequence of $k$-primitive elementary divisor tuples $\vec{\ell}_k\in \br{I_k}$
(as $k\to\infty$) we have that   
$\wt{a}(k)_*\wt{\bm{\mu}}_{k,\vec{\ell}_k}\wstar\mb{m}_{Ux_0}\times \bm{\theta}$.
\end{theorem}
\begin{proof}
We use the $\Lam_\Del$-invariance of $\wt{\bm{\mu}}_{k,\vec{\ell}_k}$ and follow Steps 1--4 as presented in~\S\ref{m disc}. The proof of Theorem~\ref{JED} presented above carries over verbatim replacing
$\wt{\bm{\mu}}_k$ with $\wt{\bm{\mu}}_{k,\vec{\ell}_k}$. 
\end{proof}
It is interesting to note that the statement of
Theorem~\ref{strong JED} is false in the cases $(m,n)=(1,1), (2,2)$. Indeed, for $(n,m)=(1,1)$ the measures
$\wt{\bm{\mu}}_{k,\vec{\ell}_k}$ are dirac masses. For $(n,m)=(2,2)$ we note that the argument presented above establishing
Theorem~\ref{JED} in this case used in~\S\ref{two two case}, in addition to the $\Lam_\Del$-invariance, also the 
invariance of $\wt{\bm{\mu}}_k$ under a $\GL_2(\bZ/k\bZ)$-action which
is no longer present when considering the $\wt{\bm{\mu}}_{k,\vec{\ell}}$'s.
In fact, if $n=m=2$ and we choose for any $k$ the $k$-primitive elementary divisors tuple to be $\vec{\ell}=(1,1)$
 then the set $\on{supp}(\wt{a}(k)_*\wt{\bm{\mu}}_{k,\vec{\ell}})=\wt{a}(k)\wt{\cP}_{k,\vec{\ell}}$, 
 which is the $\Lam_\Del$-orbit of $\Phi\pa{\smallmat{1&0\\0&1}}$ ($\Phi$ as in~\eqref{Phi}),
is contained in the fixed orbit $L_{1,1}(x_0,x_0)\subset Ux_0\times Vx_0$ (notation being as in~\eqref{Wpq}). In particular,  $\wt{a}(k)_*\wt{\bm{\mu}}_{k,\vec{\ell}}$ 
cannot be expected to converge to $\mb{m}_{Ux_0}\times\mb{m}_{Vx_0}$. This shows that the statement
of Theorem~\ref{strong JED} is simply false in this case.

\section{Proof of a strengthening of Theorem~\ref{generalized Schmidt}}\label{sec appf}

Below we state and prove a generalization of Theorem~\ref{generalized Schmidt}. Recall the notation of~\S\ref{ssecapp} and in particular,
that $n=1$ and so $d=m+1$. Theorem~\ref{strong JED} obtains a rather concrete meaning:
The set $\br{I_k}$ from~\eqref{Ik} is simply $(\bZ/k\bZ)^\times$ and the elementary divisor tuple of 
a vector $\mb{u}$ is simply $\gcd(\mb{u})$ so that for $c_k\in(\bZ/k\bZ)^\times$, the measure
$\wt{\bm{\mu}}_{k,c_k}$ is simply the normalized counting measure on the diagonal embedding of
the collection $\set{x_{k^{-1}\mb{u}}: \gcd(\mb{u})=c_k}$ in $Ux_0\times Ux_0$. 
Theorem~\ref{strong JED} states that 
$\wt{a}(k)_*\wt{\bm{\mu}}_{k,c_k}\wstar \bm{m}_{Ux_0}\times \mb{m}_{Hx_0}$ for any choice of $c_k< k$ coprime to~$k$. 
In particular, it follows that 
if $\bm{\mu}_{k,c_k}\defi (\pi_2)_*\wt{\bm{\mu}}_{k,c_k}$, then $a(k)_*\bm{\mu}_{k,c_k}\to \mb{m}_{Hx_0}$. 
We shall need the following corollary
of Theorem~\ref{strong JED}.

\begin{corollary}\label{corforSchmidt}
Let $F\subset Ux_0$ be a subset satisfying $\on{(i)}$ $\mb{m}_{Ux_0}(F)>0$, $\on{(ii)}$ $\mb{m}_{Ux_0}(\partial F)=0$ 
(where  $\partial F$ denote the boundary of $F$ in $Ux_0$), and let $f\defi\frac{1}{\mb{m}_{Ux_0}(F)}\chi_F$. 
For any sequence  $c_k$ with $1\le c_k\le k$ and $\gcd(k,c_k)=1$, 
we have $a(k)_*fd\bm{\mu}_{k,c_k}\wstar\bm{m}_{Hx_0}$.
\end{corollary}
\begin{proof}
Given a test function $\vphi\in C_c(X_d)$ we consider the function $(f\times \vphi)$ on $Ux_0\times X_d$ given by $(f\times\vphi)(x,y)=f(x)\vphi(y)$.
Although $(f\times \vphi)$ is not continuous, assumption (ii) implies that 
the points of discontinuity are of $\mb{m}_{Ux_0}\times \mb{m}_{Hx_{0}}$-measure zero. 
Using this the weak$^*$
convergence $\wt{a}(k)_*\wt{\bm{\mu}}_{k,c_k}\to \mb{m}_{Ux_0}\times \mb{m}_{Hx_{0}}$ from Theorem~\ref{strong JED}
 implies the convergence
\begin{align*}
\smallint_{X_d\times X_d}(f\times\vphi)d\wt{a}(k)_*\wt{\bm{\mu}}_{k,c_k}&\to\smallint_{X_d\times X_d}(f\times\vphi)d\pa{\mb{m}_{Ux_0}\times 
\mb{m}_{Hx_{0}}}\\
\nonumber &=\smallint_{X_d} \vphi d\mb{m}_{Hx_{0}}.
\end{align*}
We conclude 
that indeed $a(k)_*fd\bm{\mu}_{k,c_k}\wstar\mb{m}_{Hx_{0}}$ as desired.
\end{proof}
In the course of proving Theorem~\ref{generalized Schmidt} it is clearly harmless to assume  that the set $\cF$ is contained in the face
\[
\set{(u_1,\dots,u_m,1)\in\bR^d: \av{u_i}<1}\subset\partial B_\infty.
\] 
Recall that $\eta_k$ is the normalized counting measure on the set
\[
 \set{\br{\Lam_v}:v\in k\cF, v=(\mb{u}, k), \gcd(\mb{u},k)=1}.
\]
We split this set (and $\eta_k$) according to the $\Lam$-orbit of the various $\mb{u}$'s: We denote for $1\le c\le k$ with $\gcd(c,k)=1$ by $\eta_{k,c}$ the normalized counting 
measure on 
$\set{\br{\Lam_v}:v\in k\cF,\ v=(\mb{u}, k),\ 
\gcd(\mb{u})=c
},$
then $\eta_k$ is a convex combination of the $\eta_{k,c}$'s.
 We conclude that Theorem~\ref{generalized Schmidt} is implied by the following.
\begin{theorem}\label{strongapp}
For any positive integer $k$ choose $1\le c_k\le k$ with $\gcd(k,c_k)=1$. Let $\cF\subset\set{(u_1,\dots,u_m,1)\in\bR^d: \av{u_i}<1}$  be a measurable set of positive measure with boundary in $\partial B_\infty$ of  measure 0 with respect to the $m$-dimensional Lebesgue  measure on $\partial B_\infty$. 
Then, $\eta_{k,c_k}\wstar\mb{m}_{Z_m}$ as $k\to\infty$. 
\end{theorem}

\begin{proof}
Given co-prime integers $c,k$ we will use the abbreviation
$\wh{\bZ}^d_{k,c}\defi\set{(\mb{u},k)\in\bZ^d:\gcd(\mb{u})=c}.$
We need to show that for any $f\in C_c(Z_m)$, we have that 
\begin{equation}\label{nts0}
\frac{1}{\av{\wh{\bZ}^d_{k,c_k}\cap k\cF}}\sum_{v\in \wh{\bZ}^d_{k,c_k}\cap k\cF} f(\br{\Lam_v})\to\int_{Z_m} f d\mb{m}_{Z_m}\textrm{ as }k\to\infty.
\end{equation} 
Here and below we use the notation $\br{\Lam}$ to denote the projection of a lattice $\Lam\in X_m$ to $Z_m$.
Fix a function $f\in C_c(Z_m)$, $\eps>0$, and a subset~$F\subset\cF$ as in Corollary~\ref{corforSchmidt}. 
Below we think of $X_m$ (resp.\ $Z_m$) as the space of lattices (resp.\ $\SO_m(\bR)$-orbits of such) up to homothety in the hyperplane $\bR^m$ spanned by the first $m$ coordinate
axis in $\bR^d$.

We identify $\set{(u_1,\dots u_m,1)\in\bR^d: 0\le u_i< 1}$ with $Ux_0\subset X_d$
in the obvious manner; that is $(\bu,1)$ corresponds to $x_{\bu}$.  
This way we may think of $F$ as a subset of $X_d$ or $\bR^d$ at our convenience.
It follows from the discussion in Remark~\ref{rem0}\eqref{r.4.4 i.3} that if we 
denote by $p:\bR^d\to\bR^m$ the projection onto the first $m$ coordinates, then
the measure $(\pi_3)_*a(k)_*\pa{\chi_Fd\bm{\mu}_{k,c_k}}$ on $X_m$ (when normalized to be a probability measure), is the 
normalized counting measure supported on the collection
$$\set{p(\Lam_v)\in X_m:v\in\wh{\bZ}^d_{k,c_k}\cap kF}$$
(here we think of $F$ as a subset of $\partial B_\infty$ and of $kF$ as a subset of $\bR^d$).
Since we assume $F$ has positive measure, Corollary~\ref{corforSchmidt} tells us that these measures converge to $\mb{m}_{X_m}$  as $k\to\infty$,
i.e.
\begin{equation}\label{nts0.5}
\frac{1}{\av{\wh{\bZ}^d_{k,c_k}\cap kF}}\sum_{v\in\wh{\bZ}^d_{k,c_k}\cap kF}\del_{{p(\Lam_v)}}\wstar \mb{m}_{X_m} \textrm{ as }k\to\infty.
\end{equation}
The convergence in \eqref{nts0.5} is not too far from implying~\eqref{nts0}; we only need to deal with the distortion that the orthogonal projection $p$ brings into the picture. We will see that this distortion is controllable by splitting $\cF$ into finitely
many sets $F\subset\cF$ whose diameters are sufficiently small, say smaller than $\del>0$.

Choose a finite partition $\cF=\sqcup_{i=1}^\ell F_i$, where the sets $F_i$ are measurable, have boundary of Lebesgue measure zero 
(relative to $\partial B_\infty$),
and have diameter $\le \del$. 
In order to conclude the proof we will  use \eqref{nts0.5} to show that once $\del$ is chosen small enough, for any large enough 
$k$ and for any $1\le i\le \ell$, we have
\begin{equation}\label{nts1}
\bigav{\int_{Z_m} f d\mb{m}_{Z_m}-\frac{1}{\av{\wh{\bZ}^d_{k,c_k}\cap k F_i}}\sum_{v\in\wh{\bZ}^d_{k,c_k}\cap kF_i} f(\br{\Lam_v})}\le \eps.
\end{equation}
Since $\eps$ is arbitrary,   this implies \eqref{nts0}.

To this end, fix $1\le i\le \ell$ and denote $F=F_i$.
Choose a reference point $v_0\in F$ 
and let $S$ be the inverse of the restriction of the projection $p$ to the $m$-dimensional 
linear space $\set{v_0}^\perp$. 
Let $k_{v_0}\in\SO_d(\bR)$ be chosen as described in~\S\ref{ssecapp} (i.e.\ such that $k_{v_0}\set{v_0}^\perp=\bR^m$), 
and consider 
the function $\vphi\in C_c(X_m)$ defined by $\vphi(\Lam)\defi f(\br{k_{v_0}S(\Lam)})$.
Note that this definition depends on the choice of $v_0$  but is independent of 
the choice of $k_{v_0}$, and also 
that $\int_{X_m} \vphi d\mb{m}_{X_m}=\int_{Z_m} f d\mb{m}_{Z_m}$.
Applying~\eqref{nts0.5} to the function $\vphi$ we conclude that 
 \begin{equation}\label{nts2}
 \lim_k\bigav{\frac{1}{\av{\wh{\bZ}^d_{k,c_k}\cap kF}}\sum_{v\in\wh{\bZ}^d_{k,c_k}\cap kF}f(\br{k_{v_0}S(p(\Lam_v))} -\int_{Z_m} f d \mb{m}_{Z_m}}
 =0.
 \end{equation}
 Let $\on{Cone}(F)\defi \set{tF:t>0}$.
 Let $\on{d}_{Z_m}(\cdot,\cdot)$ denote a distance function on $Z_m$
and assume for the moment that it satisfies 
 \begin{equation}\label{nts11}
 \sup_{v\in \bZ^d\cap\on{Cone}(F)} \on{d}_{Z_m}\bigl(\br{\Lam_v},\br{k_{v_0}S(p(\Lam_v))}\bigr)\to 0\textrm{ as }\on{diam}(F)\to0.
\end{equation}
Using~\eqref{nts11} and the uniform continuity of $f$ we see that once  $\del>0$ is chosen
small enough, $\av{f(\br{k_{v_0}S(p(\Lam_v))})-f(\br{\Lam_v})}\le \eps$. 
Plugging this into~\eqref{nts2} gives~\eqref{nts1} as desired.

We are left to explain the validity of~\eqref{nts11}. 
We let $\on{d}_{\SL_m(\bR)}(\cdot,\cdot)$ denote a right $\SL_m(\bR)$-invariant metric on $\SL_m(\bR)$.
With it we induce metrics 
$\on{d}_{X_m}, \on{d}_{Z_m}$ on $X_m,Z_m$ respectively by the formulas 
$$\on{d}_{X_m}(g\SL_m(\bZ),h\SL_m(\bZ))=\inf_{\ga\in\SL_m(\bZ)}\on{d}_{\SL_m(\bR)}(g\ga,h),$$
$$\on{d}_{Z_m}(\br{x},\br{y})=\inf_{k_1,k_2\in\SO_m(\bR)}\on{d}_{X_m}(k_1x,k_2y)$$
for $g,h\in\SL_m(\bZ)$ and $x,y\in X_m$.
The validity of~\eqref{nts11} 
now follows as 
\begin{align*}
\on{d}_{Z_m}(\br{\Lam_v},\br{k_{v_0}S(p(\Lam_v))})&\le \on{d}_{X_m}(k_v\Lam_v,k_{v_0}S p k_v^{-1}k_v\Lam_v)\\
&\le \on{d}_{\SL_d(\bR)}(e, k_{v_0}S p k_v^{-1})
\end{align*}
and as $\on{diam}(F)\to 0$, if the matrix $k_v$ is chosen correctly,
\[
 k_{v_0}S p k_v^{-1}|_{\bR^m}:\bR^m\to\bR^m
\]
approaches the identity.
\end{proof}

\begin{acknowledgments}
We would like to thank Han Li and Jens Marklof for valuable discussions.  We further thank the Erwin Schr\"odinger Institute in Vienna and the Forschungsinstitut Mathematik at the ETH Zurich for their hospitality. Finally we thank the referee for his criticism that improved the presentation significantly.
\end{acknowledgments}

\def\cprime{$'$} \def\cprime{$'$} \def\cprime{$'$}

\end{document}